\documentclass[a4paper,11pt]{amsart}
\usepackage[english]{babel}
\usepackage{mathrsfs}
\usepackage{amsfonts}
\usepackage[centertags]{amsmath}
\usepackage{amsthm}
\usepackage{enumitem}
\usepackage{latexsym,amssymb}
\usepackage{graphicx,color}
\usepackage{mathtools}

\usepackage{comment}

\setenumerate[1]{label=(\roman*), ref=(\roman*)}
\usepackage[all]{xy}

\newtheorem{thm}{Theorem}[section]
\newtheorem{prop}[thm]{Proposition}
\newtheorem{definition}[thm]{Definition}
\newtheorem{lemma}[thm]{Lemma}

\newtheoremstyle{obs}
  {3pt}
  {3pt}
  {}
  {}
  {\bfseries}
  {.}
  {.5em}
  {}
\theoremstyle{obs}
\newtheorem{remark}[thm]{Remark}
\newtheorem{ex}[thm]{Example}

\newcommand{\R}{\mathbb{R}}      

\def\qed{\ifvmode\removelastskip\fi
{\unskip\nobreak\hfil\penalty50\hbox{}\nobreak\hfil \hbox{\vrule
height1.2ex width1.2ex}\parfillskip=0pt \finalhyphendemerits=0
\par \smallskip}}

\begin{document}

\title{The inverse problem of the calculus of variations for discrete systems}

\author[M. Barbero-Li\~n\'an, M. Farr\'{e} Puiggal\'{i}, S. Ferraro,  D. Mart\'{i}n de Diego]{Mar\'ia Barbero-Li\~n\'an$^1$, Marta Farr\'e Puiggal\'i$^2$, Sebasti\'{a}n Ferraro$^3$, \\  David Mart\'{\i}n de Diego$^2$}

\address{$^1$ 	Departamento de Matem\'atica Aplicada, ETS Arquitectura, Universidad Polit\'ecnica de Madrid, 
Avd. Juan de Herrera 4,  Madrid, 28040, Spain  
}
\email[Mar\'ia Barbero-Li\~n\'an]{m.barbero@upm.es} 
\address{
$^2$  Instituto de Ciencias Matem\'aticas (CSIC-UAM-UC3M-UCM),
   Calle Nicol\'as Cabrera 13-15, Campus UAM, Cantoblanco, 
   Madrid, 28049, Spain 	
}
\email[Marta Farr\'e Puiggal{\'\i}]{marta.farre@icmat.es} \email[David Mart{\'\i}n de Diego]{david.martin@icmat.es}
\address{$^3$ Universidad Nacional del Sur, Instituto de Matem\'{a}tica Bah\'{i}a Blanca and CONICET, Av. Alem 1253
B8000CPB Bah{\'\i}a Blanca, Argentina }
\email[Sebasti\'{a}n Ferraro]{sferraro@uns.edu.ar}

\maketitle

\begin{abstract}
We develop a geometric version of the inverse problem of the calculus of variations for discrete  mechanics and constrained discrete mechanics. The geometric approach consists of using suitable Lagrangian and isotropic submanifolds. We also provide a transition between the discrete and the continuous problems and propose variationality as an interesting geometric property to take into account in the design and computer simulation of numerical integrators.

\vspace{3mm}


\vspace{3mm}

\end{abstract}

\tableofcontents

\section{Introduction}

The classical inverse problem of the calculus of variations consists in determining whether or not a given system of explicit second order differential equations (SODE)
\begin{equation*} 
\ddot{q}^i=\Gamma^i(t, q^{j},\dot{q}^{j}), \quad i,j=1,\ldots,n
\end{equation*}
is equivalent to a system of Euler-Lagrange equations
\begin{equation*} 
\dfrac{d }{d t} \, \dfrac{\partial L}{\partial \dot{q}^i}-\dfrac{\partial L}{\partial q^i}=0 \, ,
\end{equation*}
for a regular Lagrangian $L(t,q,\dot{q})$ to be determined, in which case we say that the SODE is variational. The question can be rewritten in the following form: Is it possible to find a regular matrix   $g_{ij}(t,q,\dot{q})$ (the so-called multipliers) such that the system
\begin{equation} \label{ip}
g_{ij} (\ddot{q}^j-\Gamma^j(t, q,\dot{q}))=\dfrac{d }{d t} \, \dfrac{\partial L}{\partial \dot{q}^i}-\dfrac{\partial L}{\partial q^i}
\end{equation} 
admits a regular solution $L$? This problem has a long history, which dates back to the end of the 19th century. For a historical  review see \cite{PrinceKrupkova}. The first case to be solved was the case $n=1$, which is always variational \cite{Sonin1886}. The next case, $n=2$, was solved by Douglas in \cite{Douglas}, using techniques which were difficult to extend to higher dimensional cases. There are many approaches to the problem, see for instance the characterization in terms of the existence of a Poincar\'{e}-Cartan two-form \cite{81Crampin, 84CPT}. In \cite{BFM} we gave a new characterization of the inverse problem in terms of Lagrangian submanifolds, which will be reviewed later in Section \ref{sec2}. 

In this paper we explore the inverse problem for discrete systems given by second order difference equations (SOdE). We are particularly interested in the case when such systems are numerical integrators for a continuous system, and therefore it will be common to find them written in implicit form. 
In the discrete case an implicit system of second order difference equations will be given by a submanifold $M\subset Q\times Q\times Q$. Assume that the submanifold $M$ can be described as the vanishing of functions $\Phi^i(q_{k-1},q_k,q_{k+1})$, $i=1,\ldots,n$, such that the matrix $\left(\frac{\partial \Phi}{\partial q_{k+1}}\right)$ is regular.
Then a natural discrete formulation of the classical inverse problem would be to ask whether or not it is possible to find a regular discrete Lagrangian $L_d:Q\times Q\longrightarrow \mathbb{R}$ such that both systems
\begin{equation*}
\Phi(q_{k-1},q_k,q_{k+1})=0 \quad \mbox{and} \quad D_1L_d(q_k,q_{k+1})+D_2L_d(q_{k-1},q_k)=0
\end{equation*}
admit the same solutions.
A different version of the problem, which is concerned with the equality
\[
\Phi(q_{k-1},q_k,q_{k+1})=D_1L_d(q_k,q_{k+1})+D_2L_d(q_{k-1},q_k)
\]
has been addressed in \cite{BC2013,CracOp1996,HM2004}.

We regard this paper as a first step to introduce variationality as an important geometric property to detect in the study of the qualitative and quantitative behaviour of different numerical methods, mainly for constrained systems. If a system of second order differential equations could admit a variational description then  automatically it inherits some geometric properties, for instance preservation of energy and symplecticity. Of course, a good geometric integrator should take into account these preservation properties.

The paper is organized as follows. 
In Section \ref{sec2} we recall two different characterizations of the variationality  of a continuous SODE, one in terms of the existence of a Legendre transformation and one in terms of the existence of a Poincar\'{e}-Cartan two-form. 
In Section \ref{appendix} we provide an implicit version of one of these characterizations of variational SODEs, namely the one in terms of the existence of a Legendre transformation, given in \cite{BFM}. This section is included because such results are not available in the literature to our best knowledge.
In Section \ref{sec3} we introduce discrete mechanics and provide definitions of discrete variational SOdEs both in the explicit and implicit cases (in Sections \ref{explicit} and \ref{implicit} respectively). We also prove discrete analogues of the characterizations of variationality already known in the continuous case, and recalled in Section \ref{sec2}. 
In Section \ref{sec4} we show the transition between a discrete variational SOdE and a continuous variational SODE in both directions. 
In Section \ref{sec5} we show how the existence of two alternative Lagrangian formulations for a discrete SOdE can lead to constants of motion, in a similar way to the continuous case.
In Section \ref{sec6} we introduce the notion of discrete variational SOdE with constraints, replacing Lagrangian submanifolds by isotropic ones, and we prove an analogue of the characterization given in terms of the Poincar\'{e}-Cartan two-form. 
Since we have a notion of constrained variational second order system, both in the continuous and discrete cases, 
we expect that keeping the variational property from a continuous variational SODE to its discretization will be advantageous. 
For example, the rolling disk provides a constrained variational SODE. By an appropriate choice of discretization of the Lagrangian function and the constraints  we can obtain a discrete variational SOdE from the DLA algorithm in~\cite{CM2001}. This implies that we can perform an extension of the isotropic submanifold to a Lagrangian one as in~\cite{BFM}, and obtain energy functions that are approximately preserved by the discrete flow using backward error analysis.

\section{The inverse problem of the calculus of variations} \label{sec2}

A SODE $\Gamma$ on a tangent bundle $TQ$ is a  vector field $\Gamma \in {\mathfrak X}(TQ)$ such that $T\tau_Q ( \Gamma(v_q))=v_q$ for all $v_q\in T_qQ$, where $\tau_Q:TQ\rightarrow Q$ is the canonical projection. In local coordinates $(q,\dot{q})$ on $TQ$, the integral curves of $\Gamma$,
\[
c: I \to TQ, \; \; \; t \in I \to c(t) = (q^{i}(t), {\dot{q}}^i(t)) \, , 
\]
satisfy the system of differential equations 
\[
\displaystyle \frac{dq^{i}}{dt} = \dot{q}^i, \quad \frac{d\dot{q}^{i}}{dt} = \Gamma^{i}(q, \dot{q}) \, ,
\]
which is equivalent to the explicit system of second order differential equations 
\begin{equation}\label{qwe}
\frac{d^2q^{i}(t)}{dt^2}\,  = \Gamma^{i}\left(q(t), \frac{dq(t)}{dt} \right) \, .
\end{equation}
Locally we have $\Gamma=\displaystyle \dot{q}^i\frac{\partial}{\partial q^i}+\Gamma^{i}(q, \dot{q})\frac{\partial}{\partial \dot{q}^{i}}$.

An example of SODE arises from a system of Euler-Lagrange equations, which can be defined from a Lagrangian function $L:TQ\longrightarrow \mathbb{R}$ using Hamilton's variational principle, see more details in~\cite{foundation}. Consider the action functional
\[
\begin{array}{rrcl}
{\mathcal J}:&{\mathcal C}^2(q_0, q_1, [a, b])&\longrightarrow&\R\\
&c&\longmapsto & \displaystyle \int_a^b L(c(t),\dot{c}(t))\; dt \, ,
\end{array}
\] 
where 
\[
	{\mathcal C}^2(q_0, q_1, [a, b])=\left\{ c: [a, b]\subseteq {\mathbb R}\longrightarrow Q\; \Big|\; c\in C^2([a, b]),\; c(a)=q_0,\; c(b)=q_1\right\}.
\]

Hamilton's variational principle states that the critical points of ${\mathcal J}$ are the trajectories of the Lagrangian system, which coincide with the solutions of the Euler-Lagrange equations
\begin{equation} \label{ELeqns}
\frac{d}{dt}\left(\frac{\partial L}{\partial \dot{q}^i}\right)-\frac{\partial L}{\partial q^i}=0, \quad 1\leq i\leq n=\dim Q \, .
\end{equation}

If the Lagrangian is regular, that is, $\left(\frac{\partial^2 L}{\partial \dot{q}^i \partial\dot{q}^j } \right)$ is regular, then (\ref{ELeqns}) can be written as a SODE on $TQ$. In this case we can also use the Legendre transformation to obtain the corresponding Hamiltonian formulation of the problem.

\begin{definition}
	Let L be a Lagrangian function on $TQ$. The fiber derivative
	\[
	\begin{array}{cccc}
	\operatorname{Leg}_L : & TQ & \longrightarrow & T^*Q \\
	& v_q & \longmapsto & \operatorname{Leg}_L (v_q)  \, ,
	\end{array}
	\]
	defined by 
	\[\left\langle \operatorname{Leg}_L(v_q), w_q \right\rangle = \left.\frac{d}{dt}\right|_{t=0} L(v_q+tw_q)
	\] 
	is known as the Legendre transformation of $L$. Locally $\operatorname{Leg}_L(q^i,\dot{q}^j)= (q^i,p_j=\frac{\partial L}{\partial \dot{q}^j})$. 
\end{definition} 

If the Legendre transformation is a local diffeomorphism then the Lagrangian is regular.
A SODE $\Gamma$ is variational if the second order differential equations (\ref{qwe}) are  equivalent to the  Euler-Lagrange equations (\ref{ELeqns}) for some regular Lagrangian $L: TQ\rightarrow \R$.

Let $\omega_{Q}$ be the canonical symplectic form on $T^*Q$, $d_T\omega_Q=-\flat_{\omega_Q}^* (\omega_{T^*Q})$ and $\flat_{\omega_Q}:\mathfrak{X}(T^*Q) \longrightarrow \Omega^1(T^*Q)$ denote the contraction map $\flat_{\omega_Q}(X)=i_X\omega_Q$. The inverse map of $\flat_{\omega_Q}$, if exists, it is denoted by $\sharp_{\omega_Q}$. 

In \cite[Theorem 4.2]{BFM} we provided the following result, which gives a characterization of the inverse problem in terms of the existence of a Legendre transformation.

\begin{thm} \label{cont-lagrangian}
A SODE $\Gamma$ on $TQ$ is variational if and only if there exists a  local diffeomorphism  $F:TQ\longrightarrow T^{*}Q$ of fibre bundles over $Q$ such that 
$\operatorname{Im}(TF\circ\Gamma)$ is a Lagrangian  submanifold of the symplectic manifold  $(TT^{*}Q, d_T\omega_{Q})$.
\[
\xymatrix{
	TTQ \ar[rr]^{TF} && TT^{*}Q \ar[dd]^{\tau_{T^*Q}}  \\
	&&\\
	TQ \ar[uu]^{\Gamma} \ar[rr]^{F} \ar[uurr]^{TF\circ\Gamma} &&T^{*}Q  \\
}
\]
\end{thm}

Let $V(TQ)$ denote the set of all vertical vector fields for $\tau_{Q}\colon TQ \to Q$, that is, $V(TQ)=\operatorname{Ker}(T\tau_Q)$.
An alternative characterization of variational SODE was given in \cite{81Crampin} in terms of the existence of a Poincar\'{e}-Cartan two-form, see also \cite{marmo1980}. The precise result is the following. 
\begin{thm} \label{crampin}
A SODE $\Gamma$ on $TQ$ is variational if and only if there exists a two-form $\Omega$ on
$TQ$ of maximal rank such that
\begin{enumerate}
\item $d\Omega=0$,
\item $\Omega(v_{1},v_{2})=0$ for all $v_{1}, v_{2}\in V(TQ)$,
\item $\mathcal{L}_{\Gamma}\Omega=0$.
\end{enumerate}
\end{thm}

In Section \ref{explicit} we will provide an analogue of this result for the discrete case.

\subsection{The continuous implicit case} \label{appendix}
Now we will generalize the approach from Theorem  \ref{cont-lagrangian} to the case in which the second order system is given in implicit form because, to our best knowledge, it cannot be found in the literature. In the implicit case, the SODE on $TQ$ is replaced by a submanifold of $TTQ$.

Let $T^{(2)}Q$ denote the second order tangent bundle of $Q$ and it can also be interpreted as a submanifold of $TTQ$ as follows
\[
T^{(2)}Q=\left\{ v\in TTQ : T\tau_Q(v)=\tau_{TQ}(v) \right\} \, .
\]
Consider now an implicit system of second order differential equations given by a submanifold $M\subset T^{(2)}Q$.  Assume $M$ is defined by the vanishing of functions 
\begin{equation} \label{eq:cont:implicit}
\Phi^i (q,\dot{q},\ddot{q})=0, \quad i=1,\ldots,n \, ,
\end{equation}
such that $C:=\left( \frac{\partial \Phi}{\partial \ddot{q}} \right)$ is regular.
We will now derive Helmholtz conditions for the problem of finding a regular Lagrangian $L$ such that the systems
\[
\Phi^i(q,\dot{q},\ddot{q})=0 \quad \mbox{ and } \quad \frac{d}{dt}\left( \frac{\partial L}{\partial \dot{q}^i}\right) - \frac{\partial L}{\partial q^i}=0 \, , \quad  i=1,\ldots,n,
\]
have the same solutions (in which case we call the system variational). 
Emulating the explicit case, we aim for a local diffeomorphism over the identity $F:TQ \rightarrow T^*Q$ such that $TF(M)\subset TT^*Q$ is a Lagrangian submanifold of ($TT^*Q,d_T\omega_Q$). If $(q,p,\dot{q},\dot{p})$ denote fibered coordinates on $TT^*Q$ then locally  $d_T\omega_Q=dq\wedge d\dot{p}+d\dot{q}\wedge dp$. 
The submanifold $TF(M)$ is locally given by
\[
\left(
q^i, F_i(q,\dot{q}),\dot{q}^i,\frac{\partial F_i}{\partial q^j}\dot{q}^j+\frac{\partial F_i}{\partial \dot{q}^j}\ddot{q}^j
\right)
\]
plus the condition $\Phi^i (q,\dot{q},\ddot{q})=0$ for all $i=1,\ldots,n$. 
If we write $\omega_{TF}=(TF)^*d_T\omega_Q$ then locally 
\begin{align*}
	\omega_{TF}&=dq^i\wedge d\left( \frac{\partial F_i}{\partial q^j}\dot{q}^j+\frac{\partial F_i}{\partial \dot{q}^j}\ddot{q}^j \right)+d\dot{q}^i\wedge dF_i \\
	&= \left( \frac{\partial^2 F_i}{\partial q^k \partial q^j}\dot{q}^j + \frac{\partial^2 F_i}{\partial q^k \partial \dot{q}^j}\ddot{q}^j  \right) dq^i \wedge dq^k \\
	&\mathrel{\phantom{=}} + \left( \frac{\partial^2 F_i }{\partial \dot{q}^k \partial q^j}\dot{q}^j + \frac{\partial F_i}{\partial q^k} +\frac{\partial^2 F_i}{\partial \dot{q}^k \partial \dot{q}^j}\ddot{q}^j - \frac{\partial F_k}{\partial q^i} \right) dq^i \wedge d\dot{q}^k \\
	&\mathrel{\phantom{=}} + \frac{\partial F_i}{\partial \dot{q}^k} dq^i \wedge d\ddot{q}^k +\frac{\partial F_i}{\partial \dot{q}^k} d\dot{q}^i \wedge d\dot{q}^k \, .
\end{align*}
The condition that $TF(M)$ be a Lagrangian submanifold of $TT^*Q$ is equivalent to the condition $(TF \circ i_M)^*d_T\omega_Q=0$ and can be written as $\omega_{TF}(X,Y)=0$ for all $X,Y\in \mathfrak{X}(M)$. Therefore we compute a local basis for $\mathfrak{X}(M)$, by imposing that $X\in \mathfrak{X}(T^{(2)}Q)$ satisfies $d\Phi(X)=0$, and we get
\[
A_i=\frac{\partial}{\partial q^i}-\frac{\partial \Phi^j}{\partial q^i}(C^{-1})^k_j\frac{\partial}{\partial \ddot{q}^k} \, , \quad
B_i=\frac{\partial}{\partial \dot{q}^i}-\frac{\partial \Phi^j}{\partial \dot{q}^i}(C^{-1})^k_j\frac{\partial}{\partial \ddot{q}^k} \, ,
\]
where $C^{-1}$ is the inverse matrix of $\left( \frac{\partial \Phi}{\partial \ddot{q}} \right)$. 
Finally the implicit Helmholtz conditions 
\[
\omega_{TF}(B_i,B_j)=0, \quad   \omega_{TF}(A_i,B_j)=0 \quad \mbox{ and } \quad \omega_{TF}(A_i,A_j)=0
\]
are respectively given by
\begin{align}
\frac{\partial F_i}{\partial \dot{q}^j}&=\frac{\partial F_j}{\partial \dot{q}^i} \, , \label{IHC1}\\
\frac{\partial^2 F_i}{\partial \dot{q}^j\partial q^k}\dot{q}^k+\frac{\partial F_i}{\partial q^j}+\frac{\partial^2 F_i}{\partial \dot{q}^j \partial \dot{q}^k}\ddot{q}^k-\frac{\partial F_j}{\partial q^i}&=\frac{\partial F_i}{\partial \dot{q}^k}\frac{\partial \Phi^r}{\partial \dot{q}^j}(C^{-1})^k_r \, , \label{IHC2} \\
\frac{\partial^2 F_i}{\partial q^j\partial q^k}\dot{q}^k+\frac{\partial^2 F_i}{\partial q^j\partial \dot{q}^k}\ddot{q}^k-\frac{\partial F_i}{\partial \dot{q}^k}\frac{\partial \Phi^r}{\partial q^j}(C^{-1})^k_r&=\frac{\partial^2 F_j}{\partial q^i\partial q^k}\dot{q}^k+\frac{\partial^2 F_j}{\partial q^i\partial \dot{q}^k}\ddot{q}^k-\frac{\partial F_j}{\partial \dot{q}^k}\frac{\partial \Phi^r}{\partial q^i}(C^{-1})^k_r \, . \label{IHC3}
\end{align}

Using the implicit function theorem to write $\ddot{q}^i=\Gamma^i(q,\dot{q})$ in appropriate neighborhoods, we have that the Helmholtz conditions  (\ref{IHC1})-(\ref{IHC3})  are equivalent to (\ref{eq:cont:implicit}) being variational.

\begin{remark}
	Notice that for the system $\Phi^j=\ddot{q}^j-\Gamma^j(q,\dot{q})$, $j=1,\ldots,n$, the matrix $C$ is the identity matrix and conditions (11)-(13) in \cite{BFM} are recovered. Those conditions were proved to be equivalent to the classical Helmholtz conditions~\cite{Douglas}.
\end{remark}

Next we will see a very simple example that clearly shows the difference between the version of the inverse problem that we are discussing now, namely the multiplier version in~\eqref{ip}, and the first version of the question raised by Helmholtz \cite{Helmholtz1887}. His question was whether or not it is possible to find a regular Lagrangian $L$ such that
\begin{equation} \label{IP-exact}
\Phi_i(q,\dot{q},\ddot{q})=\frac{d}{dt}\left( \frac{\partial L}{\partial \dot{q}^i}\right) -\frac{\partial L}{\partial q^i} \, , \quad i=1,\ldots,n\, ,
\end{equation}
where now $\Phi_i=\delta_{ij}\Phi^j$ are regarded as the components of a covector. 
He provided a set of necessary and sufficient conditions for (\ref{IP-exact}) to hold, namely,
\begin{align}
\frac{\partial \Phi_i}{\partial \ddot{q}^j}-\frac{\partial \Phi_j}{\partial \ddot{q}^i} &=0 \, ,  \label{cHC1}  \\
\frac{\partial \Phi_i}{\partial q^j}-\frac{\partial \Phi_j}{\partial q^i} -\frac{1}{2}\frac{d}{dt}\left( \frac{\partial \Phi_i}{\partial \dot{q}^j}-\frac{\partial \Phi_j}{\partial \dot{q}^i} \right) &=0 \, , \label{cHC2}  \\
\frac{\partial \Phi_i}{\partial \dot{q}^j}+\frac{\partial \Phi_j}{\partial \dot{q}^i}-\frac{d}{dt}\left( \frac{\partial \Phi_i}{\partial \ddot{q}^j}+\frac{\partial \Phi_j}{\partial \ddot{q}^i} \right) &=0 \, , \label{cHC3}
\end{align}
which are also known as Helmholtz conditions.

\begin{ex}
	Consider the system
	\begin{equation*}
	\Phi_1=e^{\ddot{x}-x}-1=0 \, , \quad \Phi_2=\ddot{y}-y=0 \, ,
	\end{equation*}
	which is clearly implicit variational in the sense that its solutions coincide with the solutions to the Euler-Lagrange equations for the regular Lagrangian $L=\frac{1}{2}\left(\dot{x}^2+\dot{y}^2+x^2+y^2 \right)$.

	Notice that, as it should be, the implicit Helmholtz conditions (\ref{IHC1})-(\ref{IHC3}) admit solutions, for instance $F_1=\frac{\partial L}{\partial \dot{x}}=\dot{x}$, $F_2=\frac{\partial L}{\partial \dot{y}}=\dot{y}$.
	
	On the other hand, the original Helmholtz conditions, which can be directly checked on $\Phi$, are not satisfied. Indeed, replacing $\Phi$ in (\ref{cHC3}), we note that 
	\begin{equation*}
	2\frac{\partial \Phi_2}{\partial \dot{x}}-2\frac{d}{dt}\frac{\partial \Phi_1}{\partial \ddot{x}} = -2\frac{d}{dt}e^{\ddot{x}-x} =  -2e^{\ddot{x}-x}(\dot{x}-\dddot{x}) 
	\end{equation*}
	does not identically vanish.

\end{ex}

\section{Discrete inverse problem} \label{sec3}

In this section we will extend the above results for variational SODEs in the continuous setting to  the discrete one.
We will consider separately the explicit case, in which the second order difference equation is given as a map 
\[
\begin{array}{cccc}
\Gamma:&Q\times Q&\longrightarrow &Q\times Q\times Q\times Q \\
       &(q_{k-1},q_k)&\longmapsto &(q_{k-1},q_k,q_k,\tilde{\Gamma}(q_{k-1},q_k)) \, ,
\end{array}			
\]
and the implicit case, in which the second order difference equation
is given as a submanifold $M$ of $Q\times Q\times Q$, satisfying some regularity condition.

First we will give a brief introduction to discrete mechanics.

\subsection{Introduction to discrete mechanics}

We will consider $Q\times Q$ as a discrete version of $TQ$ and therefore $Q\times Q\times Q \times Q$ as a discrete analogue  of $TTQ$, see \cite{marsden-west}.  Instead of curves on $Q$, the solutions are replaced by sequences of points on $Q$. If we fix some  $N\in \mathbb{N}$ then we use the notation 
\[
\mathcal{C}_d(Q)=\left\{ q_d:\left\{ k \right\}_{k=0}^N \longrightarrow Q \right\}
\]
for the set of possible solutions, which can be identified  with the manifold $Q\times \stackrel{(N+1)}{\cdots} \times Q$.
Define a functional, the discrete action map, on the space of sequences $\mathcal{C}_d(Q)$ by
\[
S_d (q_d)=\sum_{k=0}^{N-1}L_d(q_k,q_{k+1}), \quad q_d\in \mathcal{C}_d(Q) \, .
\]
If we consider variations of $q_d$ with fixed end points $q_0$ and $q_N$ and  extremize $S_d$ over $q_1,\ldots,q_{N-1}$, we obtain the discrete Euler-Lagrange equations (DEL equations for short) 
\begin{equation}
\label{Eq:DEL}
D_1 L_d (q_k,q_{k+1})+D_2L_d(q_{k-1},q_k)=0 \quad \mbox{for all} \quad k=1,\ldots,N \, ,
\end{equation}
where $D_1L_d(q_{k-1},q_k)\in T^*_{q_{k-1}}Q$ and $D_2L_d(q_{k-1},q_k)\in T^*_{q_k}Q$ correspond to 
$dL_d(q_{k-1},q_k)$ under the identification $T^*_{(q_{k-1},q_k)}(Q\times Q)\cong T^*_{q_{k-1}}Q\times T^*_{q_k}Q$.

If $L_d$ is regular, that is, $D_{12}L_d$ is regular, then we obtain a well defined discrete Lagrangian map
\[
\begin{array}{cccc}
F_{L_d}: & Q\times Q & \longrightarrow & Q \times Q \\
& (q_{k-1},q_k) & \longmapsto & (q_k,q_{k+1}(q_{k-1},q_k)) \, ,
\end{array}
\]
where $q_{k+1}$ is the unique solution of~\eqref{Eq:DEL} for the given pair $(q_{k-1},q_k)$. We can further assure that the discrete Lagrangian map is invertible so that it is possible to write $q_{k-1}=q_{k-1}(q_k,q_{k+1})$, see \cite[Theorem 1.5.1]{marsden-west}.

In this setting we can define two discrete Legendre transformations 
\[
\mathbb{F}^{+}L_{d},\mathbb{F}^{-}L_{d}:Q\times Q\longrightarrow T^*Q \, ,
\]
since each projection is equally eligible for the base point. They can be defined as
\begin{align*}
\mathbb{F}^{+}L_{d}(q_{k-1},q_{k})&= (q_{k},D_2L_d(q_{k-1},q_{k})) \, ,\\
\mathbb{F}^{-}L_{d}(q_{k-1},q_{k})&= (q_{k-1},-D_1L_d(q_{k-1},q_{k})) \, .
\end{align*}
It holds that $(\mathbb{F}^{+}L_{d})^*\omega_Q=(\mathbb{F}^{-}L_{d})^*\omega_Q=:\Omega_{L_d}$, with local expression
\[
\Omega_{L_d}(q_{k-1},q_k)=\frac{\partial^2 L_d}{\partial q_{k-1}^i \partial q_k^j}dq_{k-1}^i \wedge dq_k^j \, .
\]

We can also define the evolution of the discrete system on the Hamiltonian side, $\tilde{F}_{L_d}:T^*Q \longrightarrow T^*Q$, by any of the formulas
\[
\tilde{F}_{L_d}=\mathbb{F}^{+}L_{d}\circ (\mathbb{F}^{-}L_{d})^{-1}=\mathbb{F}^{+}L_{d}\circ F_{L_d} \circ (\mathbb{F}^{+}L_{d})^{-1}=\mathbb{F}^{-}L_{d}\circ F_{L_d} \circ (\mathbb{F}^{-}L_{d})^{-1} \, ,
\]
because of the commutativity of the following diagram:
\[
\xymatrix{ 
(q_{k-1},q_k) \ar[rr]^{F_{L_d}} \ar[dr]_{\mathbb{F}^{+}L_{d}} & & (q_{k},q_{k+1}) \ar[ld]_{\mathbb{F}^{-}L_{d}} \ar[rd]_{\mathbb{F}^{+}L_{d}} \ar[rr]^{F_{L_d}} & &  (q_{k+1},q_{k+2}) \ar[ld]_{\mathbb{F}^{-}L_{d}} \\
& (q_k,p_k) \ar[rr]_{\tilde{F}_{L_d}} & & (q_{k+1},p_{k+1}) &
}
\]
The discrete Hamiltonian map $\tilde{F}_{L_d}:(T^*Q,\omega_Q) \longrightarrow (T^*Q,\omega_Q)$ is symplectic. Therefore the submanifold 
\begin{align*}
\left( q_k, p_k, q_{k+1}, p_{k+1} \right) &= \left( q_k, \mathbb{F}^{-}L_{d}(q_k,q_{k+1}), q_{k+1}, \mathbb{F}^{-}L_{d}(q_{k+1},q_{k+2}) \right) \\
&= \left( q_k, \mathbb{F}^{+}L_{d}(q_{k-1},q_{k}), q_{k+1}, \mathbb{F}^{+}L_{d}(q_{k},q_{k+1}) \right)
\end{align*}
is Lagrangian in $(T^*Q\times T^*Q,\Omega_Q)$, where $\Omega_{Q}:=\beta_{T^*Q}^{*}\omega_{Q}-\alpha_{T^*Q}^{*}\omega_{Q}$ is a symplectic form and $\alpha_{T^*Q},\beta_{T^*Q}:T^*Q\times T^*Q \longrightarrow T^*Q$ denote the projections onto the  first and second factor respectively.

So far we have taken as the starting point a discrete Lagrangian $L_d : Q\times Q \longrightarrow \mathbb{R}$. However, if we start with a continuous Lagrangian and take an appropriate discrete Lagrangian then the DEL equations become a geometric integrator for the continuous  Euler-Lagrange system, known as a variational integrator. 
Hence, given a regular Lagrangian function $L:TQ \longrightarrow \mathbb{R}$, we define a discrete Lagrangian $L_d: Q\times Q \times \mathbb{R} \longrightarrow \mathbb{R}$ as an approximation to the action of the continuous Lagrangian. More precisely,  for a regular Lagrangian $L$, and appropriate $h,q_0,q_1$, we can define the exact discrete Lagrangian as
\[
L_d^E(q_0,q_1,h) = \int_0^h L(q_{0,1}(t),\dot{q}_{0,1}(t))dt \, ,
\]
where $q_{0,1}(t)$ is the unique solution of the Euler-Lagrange equations for $L$ satisfying $q_{0,1}(0)=q_0$ and $q_{0,1}(h)=q_1$, see \cite{hartman,MMM3}. 
Then for a sufficiently small $h$, the solutions of the DEL equations for $L_d^E$ lie on the solutions of the Euler-Lagrange equations for $L$, see \cite[Theorem 1.6.4]{marsden-west}.

In practice, $L_d^E(q_0,q_1,h)$ will not be explicitly given. Therefore we will take 
\[
L_d(q_0,q_1,h) \approx L_d^E(q_0,q_1,h)\, ,
\]
using some quadrature rule. We obtain symplectic integrators in this way, see \cite{PatrickCuell}.

\subsection{Discrete inverse problem for explicit second order difference equations}  \label{explicit}

We will consider $Q\times Q$ as a discrete version of $TQ$ and therefore $Q\times Q\times Q \times Q$ as a discrete analogue of $TTQ$ \cite{marsden-west}. The discrete second order submanifold is given by
\[
\ddot{Q}_d = \left\{ \gamma_d \in (Q\times Q) \times (Q \times Q) : \alpha_Q\circ \beta_{Q\times Q} (\gamma_d) = \beta_Q\circ \alpha_{Q\times Q} (\gamma_d) \right\} \cong Q\times Q\times Q \, ,
\]
where $\alpha_Q,\beta_Q:Q\times Q\longrightarrow Q$ are the projections onto the first and second factor respectively, and analogously for $\alpha_{Q\times Q},\beta_{Q\times Q}:(Q\times Q) \times (Q \times Q) \longrightarrow Q\times Q$.

A map $\Gamma: Q\times Q \longrightarrow \ddot{Q}_d \subset (Q\times Q)\times (Q \times Q)$ satisfying $\alpha_{Q\times Q}\circ\Gamma=\operatorname{Id}$ will be referred to as an explicit second order difference equation (SOdE for short). 

\begin{definition} 
The explicit second order difference equation $q_{k+1}=\tilde{\Gamma}(q_{k-1},q_{k})$  is variational if and only if there is a 
regular discrete Lagrangian $L_d:Q\times Q\longrightarrow \mathbb{R}$ such that both systems
\begin{equation*}
q_{k+1}=\tilde{\Gamma}(q_{k-1},q_{k}) \quad \mbox{and} \quad D_1L_d(q_k,q_{k+1})+D_2L_d(q_{k-1},q_k)=0
\end{equation*}
admit the same solutions.
\end{definition}

To avoid technical difficulties we are assuming that $(q_{k-1},q_k)$ and $(q_k,\tilde{\Gamma}(q_{k-1},q_k))$ belong to the same neighborhood where $L_d$ is defined. 

Consider first $\alpha_Q:Q\times Q\longrightarrow Q$ as playing the role of $\tau_Q:TQ\longrightarrow Q$ in the discrete case (later in Proposition \ref{prop-pr2} we will see that we could also have chosen $\beta_Q:Q\times Q\longrightarrow Q$ as a discretization). For a given explicit second order difference equation $q_{k+1}=\tilde{\Gamma}(q_{k-1},q_{k})$  and a local diffeomorphism $F:Q\times Q\longrightarrow T^{*}Q$ over the identity, we define $\gamma_{F,\Gamma}:=(F\times F)\circ \Gamma$ as shown in the following commutative diagram:

\[
\xymatrix{
Q\times Q\times Q\times Q \ar[rr]^/5pt/{F\times F} & & T^{*}Q \times T^{*}Q \ar[d]^{\alpha_{T^*Q}}\\
Q\times Q \ar[u]^{\Gamma} \ar[rr]^{F} \ar[urr]^{\gamma_{F,\Gamma}} \ar[dr]_{\alpha_Q} & & T^{*}Q \ar[dl]^{\pi_Q} \\
& Q &
}
\]
For $(q_{k-1},q_k)\in Q\times Q$ the diagram is the following:
\[
\xymatrix{
(q_{k-1},q_k,q_k,\tilde{\Gamma}(q_{k-1},q_{k})) \ar[rr]^/-14pt/{F\times F} & & (q_{k-1},F(q_{k-1},q_k),q_k,F(q_k,\tilde{\Gamma}(q_{k-1},q_k))) \ar[d]^{\alpha_{T^*Q}} \\
(q_{k-1},q_k) \ar[u]^{\Gamma} \ar[rr]^{F} \ar[urr]^{\gamma_{F,\Gamma}} & & (q_{k-1},F(q_{k-1},q_k))
}
\]
Observe that the image of $F$ is written as $(q_{k-1},F(q_{k-1},q_k))$ to stress that the base point of the covector is $q_{k-1}$.

Let $\omega_{Q}=-d\theta_{Q}$ denote the canonical symplectic form on $T^{*}Q$ and consider on $T^{*}Q\times T^{*}Q$ the symplectic form $\Omega_{Q}:=\beta_{T^*Q}^{*}\omega_{Q}-\alpha_{T^*Q}^{*}\omega_{Q}$.  

\begin{prop} \label{exp:var}
The second order difference equation $q_{k+1}=\tilde{\Gamma}(q_{k-1},q_{k})$ is variational if and only if there is a local diffeomorphism $F:Q\times Q \longrightarrow T^{*}Q$ satisfying $\alpha_Q=\pi_Q\circ F$ and such that $\operatorname{Im}(\gamma_{F,\Gamma})$ is a Lagrangian submanifold of $(T^{*}Q\times T^{*}Q,\Omega_{Q})$.
\end{prop}

\begin{proof}

Assume there is an $F$ as in the statement.
Then $\operatorname{Im}(\gamma_{F,\Gamma})$ is a submanifold of half the dimension of $T^*Q\times T^*Q$ and the isotropy condition $\gamma_{F,\Gamma}^{*}\Omega_{Q}=0$ is satisfied, since $\operatorname{Im}(\gamma_{F,\Gamma})$ is a Lagrangian submanifold.
Since 
\[
\gamma_{F,\Gamma}^{*}\Omega_{Q}=-d((\beta_{T^*Q}\circ\gamma_{F,\Gamma})^{*}\theta_{Q}-(\alpha_{T^*Q}\circ\gamma_{F,\Gamma})^{*}\theta_{Q})
\]
is an exact two-form on $Q\times Q$, by the Poincar\'{e} lemma the condition $\gamma_{F,\Gamma}^{*}\Omega_{Q}=0$ implies 
\[(\beta_{T^*Q}\circ\gamma_{F,\Gamma})^{*}\theta_{Q}-(\alpha_{T^*Q}\circ\gamma_{F,\Gamma})^{*}\theta_{Q}=dL_{d}\]
for a locally defined map $L_{d}:Q\times Q \longrightarrow \mathbb{R}$, called the discrete Lagrangian.

In local coordinates we get
\[
-F_{i}(q_{k-1},q_k)dq^{i}_{k-1}+F_{i}(q_{k},\tilde{\Gamma}(q_{k-1},q_k))dq^{i}_k=\frac{\partial L_d}{\partial q^i_{k-1}}(q_{k-1},q_k)dq^i_{k-1}+\frac{\partial L_d}{\partial q^i_k}(q_{k-1},q_k)dq^i_k \ ,
\]
that is, $D_1L_d(q_{k-1},q_k)=-F(q_{k-1},q_k)$ and $D_2L_d(q_{k-1},q_k)=F(q_{k},\tilde{\Gamma}(q_{k-1},q_k))$. In particular $F=\mathbb{F}^{-}L_{d}$ and the admissibility condition $F(q_k,q_{k+1})=F(q_k,\tilde{\Gamma}(q_{k-1},q_k))$ gives the discrete Euler-Lagrange equations $-D_{1}L_{d}(q_{k},q_{k+1})=D_{2}L_{d}(q_{k-1},q_{k})$, see \cite[Section 3.2]{discrete-implicit}.

Assume now that $q_{k+1}=\tilde{\Gamma}(q_{k-1},q_{k})$ is variational. Then we define $F(q_{k-1},q_k)=-D_1L_d(q_{k-1},q_{k})$ to get
\begin{align*}
\left\{D_1L_d(q_k,q_{k+1})+D_2L_d(q_{k-1},q_k)=0\right\} &\equiv\left\{ q_{k+1}= \tilde{\Gamma}(q_{k-1},q_k)  \right\} \\
&\equiv\left\{ -D_1L_d(q_k,q_{k+1}) = -D_1L_d(q_k,\tilde{\Gamma}(q_{k-1},q_k))  \right\} \, ,
\end{align*}
which implies 
\[
-D_2L_d(q_{k-1},q_k)=D_1L_d(q_k,\tilde{\Gamma}(q_{k-1},q_k)) \, .
\]
Here we have used the notation $\left\{X(q_{k-1},q_k,q_{k+1})=0\right\}\equiv\left\{Y(q_{k-1},q_k,q_{k+1})=0\right\}$ to denote that the solutions of the equations $X(q_{k-1},q_k,q_{k+1})=0$ and $Y(q_{k-1},q_k,q_{k+1})=0$ coincide.
Then $\operatorname{Im}(\gamma_{F,\Gamma})$ is given by
\[
\left( q_{k-1}, -D_1 L_d (q_{k-1},q_k), q_k, -D_1 L_d (q_k, \tilde{\Gamma}(q_{k-1},q_k)) \right) =  \left( q_{k-1}, -D_1 L_d (q_{k-1},q_k), q_k, D_2 L_d (q_{k-1}, q_k) \right) 
\]
which is clearly a Lagrangian submanifold of $(T^*Q\times T^*Q,\Omega_Q)$.
\end{proof}

\begin{remark} \label{remark-symplecto}
Notice that we can equivalently work with Lagrangian submanifolds of $T^*(Q\times Q)$. First consider the symplectomorphism
\[
\begin{array}{cccc}
\Psi: & (T^*(Q\times Q),\omega_{Q\times Q}) & \longrightarrow & (T^*Q\times T^*Q,\Omega_{Q}) \\
& (\alpha_{q_0},\alpha_{q_1}) & \longmapsto & (-\alpha_{q_0},\alpha_{q_1})
\end{array}
\]
and define the one-form $\Psi^{-1}\circ\gamma_{F,\Gamma}$ on $Q\times Q$. Then the variationality is equivalent to requiring that $\Psi^{-1}\circ\gamma_{F,\Gamma}$ be closed, that is, that $\operatorname{Im}(\Psi^{-1}\circ\gamma_{F,\Gamma})$ be a Lagrangian submanifold of $(T^*(Q\times Q),\omega_{Q\times Q})$.
\end{remark}

If we impose that $\operatorname{Im}(\gamma_{F,\Gamma})$ is a Lagrangian submanifold of $(T^{*}Q\times T^{*}Q,\Omega_{Q})$ for a given SOdE $\Gamma$ then we get the following conditions on $F$:
\begin{align}
\frac{\partial F_{i}}{\partial Q^{j}_{1}}(q_{k-1},q_k)&=\frac{\partial F_{j}}{\partial Q^{i}_{1}}(q_{k-1},q_k) \label{dHC1} \, , \\
\frac{\partial F_{i}}{\partial Q^{j}_{2}}(q_{k-1},q_k)+\frac{\partial F_{j}}{\partial Q^{l}_{2}}(q_k,\Gamma_{k-1,k})\frac{\partial \Gamma^{l}}{\partial q^{i}_{k-1}}&=0 \label{dHC2} \, , \\
\frac{\partial F_{i}}{\partial Q^{j}_{1}}(q_k,\Gamma_{k-1,k})+\frac{\partial F_{i}}{\partial Q^{l}_{2}}(q_k,\Gamma_{k-1,k})\frac{\partial \Gamma^{l}}{\partial q^{j}_{k}}&=\frac{\partial F_{j}}{\partial Q^{i}_{1}}(q_k,\Gamma_{k-1,k})+\frac{\partial F_{j}}{\partial Q^{l}_{2}}(q_k,\Gamma_{k-1,k})\frac{\partial \Gamma^{l}}{\partial q^{i}_{k}} \label{dHC3} \, ,
\end{align}
where $\Gamma_{k-1,k}$ is short notation for $\Gamma(q_{k-1},q_k)$ and $\partial/\partial Q_1$, $\partial/\partial Q_2$ denote partial derivatives with respect to the first and second slot respectively. When the evaluation point is $(q_{k-1},q_k)$ we will usually omit it and use $\partial/\partial q_{k-1}$, $\partial/\partial q_k$ instead of $\partial/\partial Q_1$, $\partial/\partial Q_2$, for instance $\frac{\partial \Gamma^{l}}{\partial q^{i}_{k-1}}=\frac{\partial\Gamma^{l}}{\partial Q_1^{i}}(q_{k-1},q_k)$. We will refer to these equations as \textbf{discrete Helmholtz conditions}. 

Since we are assuming $\Gamma:U\longrightarrow U\times U$, for some open subset $U\subset Q\times Q$, using (\ref{dHC1}) the last condition can be reduced to
\begin{equation}
\frac{\partial F_{i}}{\partial Q^{l}_{2}}(q_k,\Gamma_{k-1,k})\frac{\partial\Gamma^{l}}{\partial q^{j}_{k}}=\frac{\partial F_{j}}{\partial Q^{l}_{2}}(q_k,\Gamma_{k-1,k})\frac{\partial\Gamma^{l}}{\partial q^{i}_{k}} \label{dHC32} \, .
\end{equation}

Equivalently, following the Remark \ref{remark-symplecto}, the Helmholtz conditions can be written as the closedness condition $d(\Psi^{-1}\circ\gamma_{F,\Gamma})=0$.

\begin{ex}[Toy example]
Consider the second order difference equation 
\begin{equation}
x_{k+1}=2x_{k}-x_{k-1}, \quad y_{k+1}=2y_{k}-y_{k-1}\, , \label{toy}
\end{equation}
which is a discretization of the variational SODE $\ddot{x}=0$, $\ddot{y}=0$. In this case we already know a Lagrangian function for the continuous system, for instance $L=\frac{1}{2}(\dot{x}^{2}+\dot{y}^{2})$, so we define a discrete Lagrangian by 
\begin{equation} \label{toy-lagrangian}
L_{d}(q_{k-1},q_{k}):=\frac{h}{2}\left( \left( \frac{x_{k}-x_{k-1}}{h} \right)^{2}+\left( \frac{y_{k}-y_{k-1}}{h} \right)^{2} \right)\, ,
\end{equation}
where $q_{k}=(x_k,y_k)$. Then we can take $F$ to be $\mathbb{F}^{-}L_{d}$, that is
\[
F(x_{k-1},y_{k-1},x_k,y_k) = \left(x_{k-1},y_{k-1},\frac{x_{k}-x_{k-1}}{h},\frac{y_{k}-y_{k-1}}{h}\right)
\]
and $\operatorname{Im}(\gamma_{\Gamma,F})$, given by
\[
\left( x_{k-1}, y_{k-1}, \frac{x_{k}-x_{k-1}}{h},\frac{y_{k}-y_{k-1}}{h}, x_k, y_k, \frac{x_{k}-x_{k-1}}{h},\frac{y_{k}-y_{k-1}}{h} \right),
\]
is a Lagrangian submanifold of $(T^{*}Q\times T^{*}Q,\Omega_{Q})$. Therefore (\ref{toy}) is variational, according to Proposition \ref{exp:var}. Indeed the Lagrangian (\ref{toy-lagrangian}) has (\ref{toy}) as DEL equations.
\end{ex}

\begin{remark}
There is no preferred role between $q_{k-1}$ and $q_k$, therefore Proposition  \ref{exp:var} could be rewritten in terms of the existence of a local diffeomorphism $F^{+}:Q\times Q \longrightarrow T^{*}Q$ satisfying $\beta_Q=\pi_Q\circ F^{+}$, that is, $F^{+}(q_{k-1},q_{k}) = (q_{k},F^{+}(q_{k-1},q_{k}))$. Then we would get $F^{+}=\mathbb{F}^{+}L_{d}$. More precisely we have the following equivalence result.
\end{remark}

Let $\Phi_{\Gamma}:Q\times Q \longrightarrow Q\times Q$ denote the flow of $\Gamma$, that is, $\Phi_{\Gamma}(q_{k-1},q_k)=(q_k,\tilde{\Gamma}(q_{k-1},q_k))$, and let $\pi_Q:T^*Q\longrightarrow Q$ denote the canonical projection. 

\begin{prop} \label{prop-pr2}
There is a local diffeomorphism $F:(Q\times Q,\alpha_Q) \longrightarrow (T^{*}Q,\pi_Q)$ over the identity such that $\operatorname{Im}(\gamma_{F,\Gamma})$ is a Lagrangian submanifold of $(T^{*}Q\times T^{*}Q,\Omega_{Q})$ if and only if there is a local diffeomorphism $F^+:(Q\times Q,\beta_Q) \longrightarrow (T^{*}Q,\pi_Q)$ over the identity such that $\operatorname{Im}(\gamma_{F^+,\Gamma})$ is a Lagrangian submanifold of $(T^{*}Q\times T^{*}Q,\Omega_{Q})$.
\end{prop}

\begin{proof}
If $F$ as in the statement exists then we can define $F^{+}=F\circ\Phi_\Gamma$. Since $F$ is a local diffeomorphism, from condition (\ref{dHC2}) we get that $\frac{\partial \Gamma^l}{\partial q_{k-1}^i}$ is regular, that is, $\Phi_\Gamma$ is a local diffeomorphism and therefore so is $F^{+}$.
\[
\xymatrix{
Q\times Q\times Q\times Q \ar[rr]^{\Phi_\Gamma \times \Phi_\Gamma} \ar@/^2pc/[rrrr]^{F^{+}\times F^{+}} & & Q\times Q\times Q\times Q \ar[rr]^{F \times F} & &  T^*Q \times T^*Q \\
Q\times Q \ar[rr]^{\Phi_\Gamma} \ar[u]^\Gamma \ar@/_2pc/[rrrr]_{F^{+}}  & &  Q\times Q \ar[rr]^{F} \ar[u]^\Gamma  & & T^*Q
}
\]
On the other hand, if we impose that $\operatorname{Im}(F^{+}\times F^{+})\circ \Gamma$ is a Lagrangian submanifold, then the condition we obtain corresponding to the vanishing of the $dq_k \wedge dq_{k-1}$ factor is 
\[
\frac{\partial F^{+}_r}{\partial q_{k-1}^j}(q_{k-1},q_k) = \left(- \frac{\partial F^{+}_i}{\partial Q_1^r}(q_k,\Gamma_{k-1,k}) - \frac{\partial F^{+}_i}{\partial Q_2^s}(q_k,\Gamma_{k-1,k})\frac{\partial \Gamma^s}{\partial q_k^r} + \frac{\partial F^{+}_s}{\partial Q_2^i}(q_k,\Gamma_{k-1,k})\frac{\partial \Gamma^s}{\partial q_k^r}\right)\frac{\partial \Gamma^i}{\partial q_{k-1}^j} \, ,
\]
which implies that  $\Phi_\Gamma$ is a local diffeomorphism since $F^{+}$ is a local diffeomorphism. Therefore we can locally define $F=F^{+}\circ \Phi^{-1}_\Gamma$, which is also a local diffeomorphism.

Finally, from the commutativity of the above diagram, we have that $\operatorname{Im}(F^{+}\times F^{+})\circ \Gamma$ is Lagrangian if and only if $\operatorname{Im}(F\times F)\circ \Gamma$ is Lagrangian.
\end{proof}

The following result is a discrete analogue of Theorem \ref{crampin}. 

\begin{prop} \label{discrete-crampin}
An explicit second order difference equation $\Gamma:Q\times Q\longrightarrow Q\times Q\times Q\times Q$ is variational if and only if there is a nondegenerate two-form $\Omega_d$ on $Q\times Q$ such that
\begin{enumerate}
\item $\mathcal{L}^{d}_{\Gamma}\Omega_d=0\, ,$
\item $\Omega_d(V_1,V_2)=0$ for all $V_1,V_2\in \operatorname{Ker}(T\alpha_Q)\, ,$
\item $d\Omega_d=0\, ,$
\end{enumerate}
where $\mathcal{L}^{d}_\Gamma\Omega_d:=(\Phi_{\Gamma})^{*}\Omega_d-\Omega_d$ is regarded as a discrete analogue of the Lie derivative.
\end{prop}

\begin{proof}
If $\Gamma$ is variational then we can either use $F$ from Proposition \ref{exp:var} or $F^{+}$ from Proposition \ref{prop-pr2} to define the two-form $\Omega_{d}:=F^{*}\omega_{Q}=(F^+)^{*}\omega_{Q}$, which clearly satisfies \textit{(iii)} and the nondegeneracy requirement. From its coordinate expression,
\begin{align*}
dq_{k-1}^i\wedge dF_i(q_{k-1},q_k)&=\left(\frac{\partial F_i}{\partial q_{k-1}^j}\right)dq_{k-1}^i\wedge dq_{k-1}^j+\left(\frac{\partial F_i}{\partial q_k^j}\right)dq_{k-1}^i\wedge dq_k^j\\
&\stackrel{\mathclap{(\ref{dHC1})}}{=}\left(\frac{\partial F_i}{\partial q_k^j}\right)dq_{k-1}^i\wedge dq_k^j\, ,
\end{align*}
condition  \textit{(ii)} is also clear. Finally
\begin{align*}
\mathcal{L}^{d}_\Gamma\Omega_d=(\Phi_{\Gamma})^{*}\Omega_d-\Omega_d&=\left(\frac{\partial F_i (q_k, \Gamma(q_{k-1},q_k))}{\partial Q_2^j}\right)dq_{k}^i\wedge d\Gamma^j - \left(\frac{\partial F_i}{\partial q_k^j}\right)dq_{k-1}^i\wedge dq_k^j\\
&=\frac{\partial F_i}{\partial Q_2^l} (q_k, \Gamma_{k-1,k})\frac{\partial \Gamma^l}{\partial q^j_{k}}dq_{k}^i\wedge dq_{k}^j\\
&\mathrel{\phantom{=}}+\left(\frac{\partial F_i}{\partial Q_2^l} (q_k, \Gamma_{k-1,k})\frac{\partial \Gamma^l}{\partial q^j_{k-1}}+\frac{\partial F_j}{\partial Q_2^i}(q_{k-1},q_k) \right) dq_{k}^i\wedge dq^j_{k-1}=0\\
\end{align*}
since the discrete Helmholtz conditions (\ref{dHC2}) and (\ref{dHC32}) are satisfied by $F$.

Conversely, let $\Omega$ be a nondegenerate two-form on $Q\times Q$ satisfying \textit{(i)-(iii)}. From \textit{(iii)}, locally $\Omega=d\Theta$ and from \textit{(ii)} $\Theta$ has the local expression
\[
\Theta=\alpha_idq_{k-1}^i+\frac{\partial h}{\partial q_k^i}(q_{k-1},q_k)dq_k^i
\]
for a locally defined map $h:Q\times Q\longrightarrow \mathbb{R}$. Then take $\bar{\Theta}=\Theta-dh$, which satisfies $\bar{\Theta}(V)=0$ for all $V\in \operatorname{Ker}(T\alpha_Q)$ and $d\bar{\Theta}=\Omega$. Then $F:Q\times Q \longrightarrow T^*Q$ can be defined by
\[
\langle F(q_{k-1},q_k), v_{q_{k-1}} \rangle=\langle \bar{\Theta}(q_{k-1},q_k), V_{v_{q_{k-1}}} \rangle \mbox{ for all }  v_{q_{k-1}}\in TQ\, ,
\]
where $V_{v_{q_{k-1}}}\in T_{(q_{k-1},q_k)}(Q\times Q)$ is any vector satisfying $T\alpha_Q(V_{v_{q_{k-1}}})=v_{q_{k-1}}$.
The first condition,
\[
\mathcal{L}_{\Gamma}^d\Omega=(\Phi_\Gamma)^*d\bar{\Theta}-d\bar{\Theta}=d((\Phi_\Gamma)^*\bar{\Theta}-\bar{\Theta})=d\mathcal{L}_{\Gamma}^d\bar{\Theta}=0 \, ,
\] 
implies the local existence of a discrete Lagrangian $L_d:Q\times Q\longrightarrow \mathbb{R}$ such that $\mathcal{L}_{\Gamma}^d\bar{\Theta}=dL_d$. From the local expression it is clear that $\gamma_{F,\Gamma}=\Psi\circ dL_d$. As $\operatorname{Im}(dL_d)$ is a Lagrangian submanifold of $T^*(Q\times Q)$, $\operatorname{Im}(\gamma_{F,\Gamma})$ is also a Lagrangian submanifold of $T^*Q\times T^*Q$, that is, $\Gamma$ is variational. Finally, for $\Omega$ to be nondegenerate it is necessary to have $\left( \frac{\partial F}{\partial q_k} \right)$ nondegenerate, that is, $F$ is a local diffeomorphism.
\end{proof} 

\begin{remark}
The second condition in Proposition \ref{discrete-crampin} can be replaced by 
\begin{enumerate}
\item[(ii)'] $\Omega_d(V_1,V_2)=0$ for all $V_1,V_2\in \operatorname{Ker}(T\beta_Q)$,
\end{enumerate} 
which corresponds to the absence of the term $dq_{k-1}\wedge dq_{k-1}$  instead of the term $dq_k\wedge dq_k$ in $\Omega_d$.
\end{remark}

\begin{remark}[The one-dimensional case]
In the continuous one-dimensional case, that is, when we are given just one second order differential equation $\ddot{q}=\Gamma(t,q,\dot{q})$, an old result by Sonin \cite{Sonin1886} shows that a regular Lagrangian always exists. This can be proved by showing that the only Helmholtz condition that remains, which is $\frac{\partial g}{\partial t}+\dot{q}\frac{\partial g}{\partial q}+\Gamma\frac{\partial g}{\partial \dot{q}}=\frac{\partial \Gamma}{\partial \dot{q}}g$, always admits a solution $g\not=0$.

In the discrete autonomous setting the only Helmholtz condition to be studied is
\begin{equation*}
\frac{\partial F}{\partial Q_{2}}(q_{k-1},q_k)+\frac{\partial F}{\partial Q_{2}}(q_k,\Gamma(q_{k-1},q_k))\frac{\partial \Gamma}{\partial q_{k-1}}(q_{k-1},q_k)=0\, ,
\end{equation*}
that is, the problem reduces to determining whether or not the functional equation
\begin{equation} \label{func}
g(y,f(x,y))\frac{\partial f}{\partial x}(x,y)=-g(x,y)
\end{equation}
has a nonzero solution $g$ for a given map $f:\mathbb{R}^2 \longrightarrow \mathbb{R}$.

Assume $\Gamma$ is linear, that is, $q_{k+1}=aq_{k-1}+bq_k$ for some constants $a$ and $b$, with $a\not=0$. 
Does (\ref{func}) admit a solution $g\not\equiv 0$? We do not have a classification even for this linear case, but some positive examples follow.
\begin{itemize}
\item If $a=-1$, then any  constant $g$ is a solution.
\item If $a=1$, $b=0$, then $g(y,x)=-g(x,y)$ admits a solution, for instance $g(x,y)=x-y$.
\item If $a<0$, $b=0$, then $g(y,ax)a=-g(x,y)$ admits a solution $g(x,y)=\frac{1}{|xy|}$ away from $xy=0$.
\item If $a>0$, $b=0$, then $g(x,y)=\frac{1}{x|y|}-\frac{1}{y|x|}$ is a solution away from $xy\geq 0$.
\item If $a\not=0$, $b=\frac{a^3-1}{a}$, then $g(x,y)=-a^2Bx+By$ is a solution for all $B\not=0$, away from $(0,0)$.
\end{itemize}
It would be interesting to obtain a complete classification of the variationality of second order difference equations in low dimensions.

\end{remark}

\subsection{Discrete inverse problem for implicit second order difference equations} \label{implicit}

Now we go back to the implicit case, where a system of second order difference equations is given by a submanifold $M\subset Q\times Q\times Q$. We assume that $M$ is given by the vanishing of functions $\Phi^i(q_{k-1},q_k,q_{k+1})$, $i=1,\ldots,n$, such that $C:=\left(\frac{\partial \Phi}{\partial q_{k+1}}\right)$ is invertible. The problem then consists in deciding whether or not the original system is equivalent to a discrete Lagrangian system.

\begin{definition}
The implicit system of second order difference equations $\Phi^i(q_{k-1},q_k,q_{k+1})=0$, $i=1,\ldots,n$, is variational if and only if there is
a regular discrete Lagrangian $L_d:Q\times Q\longrightarrow \mathbb{R}$ such that both systems
\begin{equation*}
\Phi(q_{k-1},q_k,q_{k+1})=0 \quad \mbox{and} \quad D_1L_d(q_k,q_{k+1})+D_2L_d(q_{k-1},q_k)=0
\end{equation*}
admit the same solutions.
\end{definition}

\begin{prop} \label{prop:sode:implicit}
An implicit SOdE locally given by the vanishing of constraints
\[
\Phi^i(q_{k-1},q_k,q_{k+1})=0, \quad i=1,\ldots,n \, ,
\]
is variational if and only if there is a local diffeomorphism $F:Q\times Q\rightarrow T^*Q$ satisfying $\alpha_Q=\pi_Q\circ F$ 
and such that $\operatorname{Im}(\left. (F\times F) \right|_M)$ is a Lagrangian submanifold of $(T^*Q\times T^*Q,\Omega_Q)$, where 
\[
M=\left\{(q_{k-1},q_k,q_{k+1})\in Q\times  Q \times Q : \Phi(q_{k-1},q_k,q_{k+1})=0 \right\} \, .
\]
\end{prop}

\begin{proof}
Assume first that a local diffeomorphism $F$ with the stated properties exists. Since we have assumed that $C$ is regular, we can use the implicit function theorem to get
for each $(q_{k-1},q_k,q_{k+1})\in M$ neighborhoods $U$ of $(q_{k-1},q_k)$, $V$ of $q_{k+1}$ and $\tilde{\Gamma} : U \longrightarrow V$ such that
\begin{eqnarray*}
&& \left\{ (q_{k-1},q_k,q_{k+1})\in U\times V : \Phi(q_{k-1},q_k,q_{k+1})=0     \right\} \\
& \equiv & \left\{  (q_{k-1},q_k,q_{k+1})\in U\times V : q_{k+1}= \tilde{\Gamma}(q_{k-1},q_k)  \right\} \\
& \equiv & \left\{ (q_{k-1},q_k,q_{k+1})\in U\times V : F(q_k,q_{k+1})-F(q_k,\tilde{\Gamma}(q_{k-1},q_k))=0 \right\} \\
& \equiv & \left\{ (q_{k-1},q_k,q_{k+1})\in U\times V : D_1L_d(q_k,q_{k+1})+D_2L_d(q_{k-1},q_k)=0 \right\} 
\end{eqnarray*}
for some locally defined Lagrangian $L_d$, as explained in Section \ref{explicit}.

Now assume $\Phi(q_{k-1},q_k,q_{k+1})=0$ is variational, that is, the two sets of equations
\begin{equation*}
\Phi(q_{k-1},q_k,q_{k+1})=0 \quad \mbox{and} \quad D_1L_d(q_k,q_{k+1})+D_2L_d(q_{k-1},q_k)=0 \, ,
\end{equation*}
have the same solutions for some locally defined Lagrangian $L_d$. If we choose 
\[
F(q_{k-1},q_k)=-D_1L_d(q_{k-1},q_{k})
\]
then we have 
\begin{align*}
\left\{D_1L_d(q_k,q_{k+1})+D_2L_d(q_{k-1},q_k)=0\right\}&\equiv\left\{\Phi(q_{k-1},q_k,q_{k+1})=0\right\} \\
\equiv\left\{ q_{k+1}= \tilde{\Gamma}(q_{k-1},q_k)  \right\} &\equiv \left\{ -D_1L_d(q_k,q_{k+1}) = -D_1L_d(q_k,\tilde{\Gamma}(q_{k-1},q_k))  \right\}
\end{align*}
which implies 
\[
-D_2L_d(q_{k-1},q_k)=D_1L_d(q_k,\tilde{\Gamma}(q_{k-1},q_k)) \, .
\]
Then $(F\times F)(M)$ is locally given by
\[
\left( q_{k-1}, -D_1 L_d (q_{k-1},q_k), q_k, -D_1 L_d (q_k, \tilde{\Gamma}(q_{k-1},q_k)) \right) =  \left( q_{k-1}, -D_1 L_d (q_{k-1},q_k), q_k, D_2 L_d (q_{k-1}, q_k) \right) 
\]
which is clearly Lagrangian in $(T^*Q\times T^*Q,\Omega_Q)$.
\end{proof}

In \cite{BFM} the Helmholtz conditions for explicit SODEs are derived using Lagrangian submanifolds, as recalled at the beginning of Section \ref{sec2}. For an implicit SODE we can also derive Helmholtz conditions using Lagrangian submanifolds, as described in Section \ref{appendix}. Analogously now we can obtain the implicit discrete  Helmholtz conditions.

The submanifold $(F\times F)(M)$ is locally given by
\[
\left(
q_{k-1},F(q_{k-1},q_k),q_k,F(q_k,q_{k+1})
\right)
\]
plus the condition $\Phi^i(q_{k-1},q_k,q_{k+1})=0$ for all $i=1,\ldots,n$.
If we write $\tilde{\omega}=(F\times F)^*\Omega_Q $ then locally 
\begin{align*}
\tilde{\omega}&= dF_i(q_k,q_{k+1})\wedge dq_k^i-dF_i(q_{k-1},q_k)\wedge dq_{k-1}^i \\
&= \frac{\partial F_i}{\partial Q_1^j}(q_k,q_{k+1}) dq_k^j \wedge dq_k^i +\frac{\partial F_i}{\partial Q_2^j}(q_k,q_{k+1}) dq_{k+1}^j \wedge dq_k^i \\
&\mathrel{\phantom{=}} -\frac{\partial F_i}{\partial Q_1^j}(q_{k-1},q_k) dq_{k-1}^j \wedge dq_{k-1}^i - \frac{\partial F_i}{\partial Q_2^j}(q_{k-1},q_k)dq_k^j \wedge dq_{k-1}^i \, .
\end{align*}
Then the condition that $(F\times F)(M)$ be a Lagrangian submanifold of $(T^*Q\times T^*Q,\Omega_Q)$ is equivalent to the condition $((F\times F) \circ i_M)^*\Omega_Q=0$ and can be written as $\tilde{\omega}(X,Y)=0$ for all $X,Y\in \mathfrak{X}(M)$. Therefore we will compute a local basis for $\mathfrak{X}(M)$. By imposing that $X\in \mathfrak{X}(Q\times Q\times Q)$ satisfies $d\Phi(X)=0$ we get
\[
A_i=\frac{\partial}{\partial q_{k-1}^i}-\frac{\partial \Phi^j}{\partial q_{k-1}^i}(C^{-1})^r_j\frac{\partial}{\partial q_{k+1}^r} \, , \quad
B_i=\frac{\partial}{\partial q_k^i}-\frac{\partial \Phi^j}{\partial q_k^i}(C^{-1})^r_j\frac{\partial}{\partial q_{k+1}^r} \, ,
\]
where $C^{-1}$ denotes the inverse matrix of $C=\left(\frac{\partial \Phi}{\partial q_{k+1}}\right)$.
Finally the \textbf{implicit discrete Helmholtz conditions} 
\[
\tilde{\omega}(A_i,A_j)=0, \quad   \tilde{\omega}(A_i,B_j)=0 \quad \mbox{ and } \quad \tilde{\omega}(B_i,B_j)=0
\]  
are respectively given by
\begin{align*}
\frac{\partial F_i}{\partial Q_1^j}(q_{k-1},q_k) &= \frac{\partial F_j}{\partial Q_1^i}(q_{k-1},q_k) \, , \\
\frac{\partial F_j}{\partial Q_1^i}(q_k,q_{k+1})+\frac{\partial F_j}{\partial Q_2^k}(q_k,q_{k+1})\frac{\partial \Phi^r}{\partial q_k^i}(C^{-1})^k_r &= \frac{\partial F_i}{\partial Q_1^j}(q_k,q_{k+1})+\frac{\partial F_i}{\partial Q_2^k}(q_k,q_{k+1})\frac{\partial \Phi^r}{\partial q_k^j}(C^{-1})^k_r   \, , \\
\frac{\partial F_i}{\partial Q_2^j}(q_{k-1},q_k)&=\frac{\partial F_j}{\partial Q_2^k}(q_k,q_{k+1})\frac{\partial \Phi^r}{\partial q_{k-1}^i}(C^{-1})^k_r  \, .
\end{align*}

\begin{remark}
If an implicit system is variational  then it is possible to find functions $g_{ij}(q_{k-1},q_k,q_{k+1})$ such that
\[
g_{ij}(q_{k-1},q_k,q_{k+1})\Phi^j(q_{k-1},q_k,q_{k+1})=\left[D_1L_d(q_k,q_{k+1})+D_2L_d(q_{k-1},q_k)\right]_i =:G_i(q_{k-1},q_k,q_{k+1}) \, ,
\]
as shown for instance in \cite{teitelboim}. Indeed, since $\left(\frac{\partial \Phi}{\partial q_{k+1}}\right)$ is regular, we can consider a coordinate change
\[
(q_{k-1},q_k,q_{k+1}) \longrightarrow (q^i_{k-1}, q^i_k, y^i:=\Phi^i(q_{k-1},q_k,q_{k+1})) \, ,
\]
for which now the submanifold $M$  defines the second order difference equation $y^i=0$.
If we let $G_j(q^i_{k-1},q^i_k,y^i)$ denote the function $G_j$ expressed in the new coordinates, then we have
\[
G_j(q^i_{k-1},q^i_k,y^i)=\int_0^1 \frac{d}{dt}G_j(q_{k-1},q_k,ty)dt=y^i\overbrace{\int_0^1\frac{\partial G_j}{\partial y^i}(q_{k-1},q_k,ty)dt}^{g_{ij}(q_{k-1},q_k,q_{k+1})}=\Phi^i g_{ij} \, .
\]
\end{remark}

\section{Correspondence between continuous and discrete variational SODEs} \label{sec4}

In this section, we will analyze the relationship between the inverse problem for discrete variational calculus and the inverse problem of the calculus of variations in the continuous setting.

\subsection{From the continuous to the discrete setting}
First we start from a continuous SODE $\Gamma$ on $TQ$. If it is variational then we will associate it with a discrete variational SOdE.
Denote by $\Phi^{\Gamma}$ the flow of  $\Gamma$,
\[
\Phi^{\Gamma}: U\subseteq \mathbb{R} \times TQ \to TQ \, ,
\]
where $U$ is an open subset of $\mathbb{R} \times TQ$. 
For expository simplicity we will assume that $\Gamma$ is complete and $U=\mathbb{R} \times TQ$. We will use the notation $\Phi^{\Gamma}_t(v_q)=\Phi^{\Gamma}(t, v_q)$.

\begin{prop}\label{intermediate} A complete SODE $\Gamma$ on $TQ$ is variational if and only if there is a local diffeomorphism  $F: TQ\rightarrow T^*Q$ of fiber bundles over $Q$  such that $\operatorname{Im}(F\times F)\circ(\operatorname{Id}\times \Phi_t^{\Gamma})$ is a Lagrangian submanifold of $(T^{*}Q \times T^*Q, \Omega_Q=\beta_{T^*Q}^*\omega_Q-\alpha_{T^*Q}^*\omega_Q)$ for all $t\in \R$.
\end{prop}
\begin{proof}
According to the characterization given in \cite{BFM} (and recalled in Section \ref{sec2}), $\Gamma$ is variational if and only if there exists a local diffeomorphism  $F: TQ\rightarrow T^*Q$ of fiber bundles over $Q$  such that 
$\operatorname{Im}(TF\circ\Gamma)$ is a Lagrangian submanifold of $(TT^{*}Q, d_T\omega_{Q})$. This is equivalent to the condition 
\[
\mathcal{L}_{\Gamma}F^*\omega_Q=0
\]
and therefore
\[
(\Phi_t^{\Gamma})^*(F^*\omega_Q)=F^*\omega_Q \, .
\]
The last equality is equivalent to the statement that $\operatorname{Im}(F\times F)\circ(\operatorname{Id}\times \Phi_t^{\Gamma})$ is a Lagrangian submanifold of $(T^{*}Q \times T^*Q, \Omega_Q)$.
\[
\xymatrix{
TQ\times TQ \ar[rrr]^{F\times F} &&& T^{*}Q \times T^*Q\ar[d]^{\alpha_{T^*Q}}  \\
TQ \ar[u]^{\operatorname{Id}\times \Phi_t^{\Gamma}} \ar[rrr]^{F} \ar[urrr]^{F\times (F \circ \Phi_t^{\Gamma}) \quad } &&&T^{*}Q \\
}
\]
\end{proof}

Now in order to define a variational discrete second order system, we need to introduce the exponential map associated with a second order differential equation $\Gamma$. Given a point $q_0\in Q$ and a positive real number $h_{0}$, we define
\[
\operatorname{exp}^{\Gamma}_{(h_{0}, q_{0})}(v_{q_0}) = \tau_{Q} ((
\Phi_{h_0}^{\Gamma})(v_{q_0})), \; \; \mbox{ for } v_{q_0}\in
T_{q_0}Q, 
\]
(assuming that $\Gamma$ is complete). 
For $h_0$ small enough it is possible to find  an open subset ${\mathcal U}  \subseteq TQ$ and an open subset $U$ of $Q$, with $q_{0} \in
U$ such that the map
\[
\begin{array}{rrcl}
\operatorname{exp}^{\Gamma}_{h_{0}}:& {\mathcal U}& \longrightarrow&U \times U \subseteq Q \times Q\\
& v &\longmapsto & (\tau_{Q}(v), \operatorname{exp}^{\Gamma}_{h_{0}}(v))
\end{array}
\]
is a diffeomorphism (see \cite{MMM3} for details).
Denote by 
$R^{e^-}_{h_{0}}: U \times U \to {\mathcal U}$
the inverse map of the diffeomorphism 
$\operatorname{exp}^{\Gamma}_{h_{0}}: {\mathcal U}  \to U \times U $.

\begin{thm} \label{thm-cont-disc}
Given a SODE $\Gamma$ on $TQ$, define the discrete second order difference equation
\[
\begin{array}{rrcl}
\Gamma_d:& U\times U&\longrightarrow &U\times U\times U\times U \\
 &  (q_{k-1}, q_k)&\longmapsto&(q_{k-1}, q_k, q_k,  (\tau_Q\circ \Phi_{2h}^{\Gamma}\circ R^{e^-}_{h})(q_{k-1}, q_k)) \, .
\end{array}
\]
If $\Gamma$ is variational then so is $\Gamma_d$ for $h$ small enough.
\end{thm}
\begin{proof}
Define $F_d: U\times U\subseteq Q\times Q\longrightarrow T^*Q$ by $F_d=F\circ R_h^{e^-}$, where $F$ exists from Theorem \ref{cont-lagrangian} if $\Gamma$ is variational. Then the proof is a consequence of the commutativity of the following diagram
\[
\xymatrix{
Q\times Q\times Q\times Q\ar[rrr]^{R^{e^-}_{h}\times R^{e^-}_{h}} &&&TQ\times TQ \ar[rrr]^{F\times F} &&& T^{*}Q \times T^*Q\ar[d]^{\alpha_{T^*Q}}  \\
Q\times Q\ar[rrr]^{R^{e^-}_{h}}\ar[u]^{\Gamma_d}  \ar@/_{10mm}/[rrrrrr]_{F_d}
&&&TQ \ar[u]^{\operatorname{Id}\times \Phi_h^{\Gamma}} \ar[rrr]^{F} \ar[urrr]^{F\times (F \circ \Phi_h^{\Gamma}) \quad } &&&T^{*}Q  \\
}
\]
and Proposition \ref{intermediate}, taking into account that 
\begin{align*}
(q_k, (\tau_Q\circ \Phi_{2h}^{\Gamma}\circ R^{e^-}_{h})(q_{k-1}, q_k))&=(\operatorname{exp}^{\Gamma}_{h}\circ \Phi_{h}^{\Gamma}\circ R^{e^-}_{h})(q_{k-1}, q_k)\\
&=((R^{e^-}_{h})^{-1}\circ \Phi_{h}^{\Gamma}\circ R^{e^-}_{h})(q_{k-1}, q_k)\, .
\end{align*}
\end{proof}

\subsection{From the discrete  to the continuous setting}\label{fromdc}

One of the most powerful techniques to understand the qualitative behavior of numerical methods is  backward error analysis, see \cite{hairer,serna}. The idea is that a symplectic integrator applied to a Hamiltonian system can be studied using the existence of the modified differential equation, which is a Hamiltonian system whose trajectories are arbitrarily close to the ones of the integrator.
We will turn to the Lagrangian side in this section and will find a continuous Lagrangian system which is arbitrarily close to a given variational SOdE.

Given a variational SOdE inducing a discrete flow
\[
\Psi_h(q_{k-1}, q_k)=(q_k, q_{k+1})
\]
depending on a small parameter $h>0$, is it possible to construct a continuous Lagrangian $L_h: TQ\rightarrow {\mathbb R}$, depending on $h$, such that the solutions $q: I \rightarrow Q$ of the corresponding Euler-Lagrange equations satisfy
\[
\Psi_h(q((k-1)h), q(kh))=(q(kh), q((k+1)h))\; ?
\]
This question cannot be solved with such generality~\cite{BenettinGiorgilli,CMS,serna}, but we will study a related problem. 
The idea is to approach as much as possible the discrete flow $\Psi_h$ by the map induced by a continuous Lagrangian. Since we start from a variational SOdE, we have a corresponding discrete Lagrangian $L_d^h:Q\times Q \longrightarrow \mathbb{R}$. Our aim is to find, using methods of backward error analysis, a continuous Lagrangian function such that the corresponding exact discrete Lagrangian is close to $L_d^h$ up to any chosen order of accuracy. 
This implies that the flow of the corresponding continuous SODE is also close to the original discrete system up  to the same order of accuracy, see \cite{PatrickCuell}.

Consider a submanifold $M\subset  Q\times Q\times Q$ of dimension $2n$ transverse to the fibers of the  projection 
\[
\begin{array}{rrcl}
pr_{12}: & Q\times Q\times Q&\longrightarrow& Q\times Q\\
&(q_{k-1}, q_k, q_{k+1})&\longmapsto&(q_{k-1},q_k) \, .
\end{array}
\]
The transversality condition implies that $M$ can be represented in a neighborhood $U$ of each point $(q_{k-1}, q_k, q_{k+1})\in M$ as a section $\Gamma_d: U\subset Q\times Q\longrightarrow M\subset  Q\times Q\times Q$.

Assume now that we have a one-parameter family of transverse submanifolds $M_h$, smoothly depending on a small parameter $h\in (0, t)$, $t>0$.
Under these conditions, we can construct a smooth family of sections
\[
\Gamma^h_d: U_h\subset Q\times Q\longrightarrow M_h \, ,
\]
where $U_h$ is an open subset of $Q\times Q$, such that $\Gamma^h_d(q_{k-1}, q_{k})=(q_{k-1}, q_k, \tilde{\Gamma}^h_d(q_{k-1}, q_k))$. Denote by
\[
\begin{array}{rrcl}
\epsilon_d:& Q\times Q\times Q & \longrightarrow & Q\times Q\times Q\times Q\\
           &(q_{k-1}, q_k, q_{k+1})&\longmapsto&(q_{k-1}, q_k, q_k, q_{k+1})
\end{array}
\]
the canonical inclusion.

Assume that it is possible to find a family of local diffeomorphisms $F_d^h: U_h\rightarrow T^*Q$ with $\alpha_Q=\pi_Q \circ F_d^h$, smoothly depending on $h$, and such that
\[
\Sigma_h:=(F_d^h\times F_d^h)(\epsilon_d(M_h))
\]
is a family of Lagrangian submanifolds of  $(T^{*}Q \times T^*Q, \Omega_Q=\beta_{T^*Q}^*\omega_Q-\alpha_{T^*Q}^*\omega_Q)$.
In particular each $\Gamma^h_d$ will be variational.

Each Lagrangian submanifold $\Sigma_h$ defines a symplectomorphism $\Psi_{F_d^h, M_h}(q_{k-1}, p_{k-1}) = (q_k, p_k)$ on $T^*Q$ 
implicitly given by
\begin{align}
p_{k-1}&=F_d^h(q_{k-1}, q_k) \, , \label{eqo1} \\
p_k&= F_d^h(q_k, \tilde{\Gamma}^h_d(q_{k-1}, q_{k})) \, .  \label{eqo2}
\end{align}
Locally, this means that there exists a Lagrangian function $L_d(h, q_{k-1}, q_{k})=L^h_d(q_{k-1}, q_k)$  such that the previous equations are  written as
\begin{align*}
p_{k-1}&=-D_1L^h_d(q_{k-1}, q_k) \, , \\ 
p_k&=D_2L_d^h(q_{k-1}, q_k)  \, . 
\end{align*}
Notice that $\hbox{graph}(\Psi_{F_d^h, M_h})=\Sigma_h$. In the sequel, we will make the assumption
\begin{equation*}
\lim_{h\rightarrow 0^+}\Psi_{F_d^h, M_h}=\operatorname{Id}_{T^*Q} \, ,
\end{equation*}
which means that $\Psi_{F_d^h, M_h}$ is close to the identity for small $h$.
Therefore it is possible to rewrite Equations (\ref{eqo1}) and (\ref{eqo2}) as
\begin{align}
p_{k-1}&=p_{k}+\frac{\partial S_{(h)}}{\partial q_{k-1}}(q_{k-1}, p_k) \, , \label{pois1}\\
q_{k}&=q_{k-1}+\frac{\partial S_{(h)}}{\partial p_k}(q_{k-1}, p_k) \, ,\label{pois2}
\end{align}
for a type 2  generating function $S_{(h)}$ whose related type 1 generating function is precisely $L_d^h$, see \cite[Section VI.5.1]{hairer}.

Define  $S_{(0)}:=0$ and consider  a formal expansion of $S_{(h)}$  in powers of  $h$  
\[
S_{(h)}(q_{k-1}, p_k)= hS_1(q_{k-1}, p_k)+h^2S_2(q_{k-1}, p_k)+h^3S_3(q_{k-1}, p_k)+\ldots
\]

For any symplectic method there is formally a Hamiltonian system, known as the modified differential equation, such that the corresponding flow, evaluated at the time step, is precisely the given method. The Hamiltonian function appears as a formal series and for a rigorous treatment, this series has to be truncated.

\begin{thm}[\cite{BenettinGiorgilli,hairer}] \label{thm:hairer}
Assume that  the symplectic method $\Psi_{F_d^h, M_h}: T^*Q\rightarrow T^*Q$ has a generating function 
\[
S_{(h)}(q_{k-1}, p_k)= hS_1(q_{k-1}, p_k)+h^2S_2(q_{k-1}, p_k)+h^3S_3(q_{k-1}, p_k)+\ldots
\]
with smooth $S_j(q_{k-1},p_k)$, defined on an open set $\bar{U}$. Then, the modified differential equation is a Hamiltonian system with 
\begin{equation} \label{Hamiltonian:series}
H_h(q,p)=H(q, p)+hH_2(q, p)+h^2H_3(q, p)+\ldots
\end{equation}
where the functions $H_j(q_{k-1},p_k)$ are defined and smooth on $\bar{U}$.
\end{thm}

The series (\ref{Hamiltonian:series}) may not converge, but if we consider truncations 
\[
H^{(\rho)}_h(q, p)=H(q, p)+hH_2(q, p)+\ldots+ h^{\rho-1} H_{\rho}(q, p)
\]
of $H_h(q,p)$ then for each $\rho\in \mathbb{N}$ we have 
\[
\Psi_{F_d^h, M_h}-\Phi^{{H^{(\rho)}_h}}_h=\mathcal{O}(h^{\rho+1}) \, ,
\] 
where $\Phi^{H^{(\rho)}_h}_h$ is the flow of the Hamiltonian vector field $X_{H^{(\rho)}_h}$ at time $h$, see \cite[Chapter 10]{serna} and references therein.

Observe that the type 2 generating function $S_{(h)}^{\rho}$ corresponding to the Lagrangian submanifold $\hbox{graph}(\Phi^{H^{(\rho)}_h}_h)$ satisfies 
\begin{equation}\label{gener2}
S_{(h)}^{\rho}-S_{(h)}={\mathcal O}(h^{\rho+1}) \, ,
\end{equation}
since the terms of the Hamiltonian $H_h^{(\rho)}$ and the corresponding type 2 generating function  
series 
\[
S^{\rho}_{(h)}(q_{k-1}, p_k)= hS^{\rho}_1(q_{k-1}, p_k)+h^2S^{\rho}_2(q_{k-1}, p_k)+h^3S^{\rho}_3(q_{k-1}, p_k)+\ldots
\]
are related by derivatives of the Hamiltonian, see \cite{hairer}. For instance we have
\begin{align*}
S^{\rho}_1(q,p)&=H^{(\rho)}_h(q, p) \, ,\\
S^{\rho}_2(q,p)&=\frac{1}{2}\frac{\partial H^{(\rho)}_h}{\partial q}(q, p)\cdot \frac{\partial H^{(\rho)}_h}{\partial p}(q, p) \, ,\\
&\cdots
\end{align*}
which are obtained by expanding the  Hamilton-Jacobi equation corresponding to generating functions of type 2, that is,
\[
\frac{\partial S^{\rho}}{\partial t}(h,q,p)=H^{(\rho)}_h\left(q+\frac{\partial S^{\rho}_{(h)}}{\partial p}(q,p), p\right)\; ,
\qquad  S^{\rho}_{(0)}(q,p)=0 \, ,
\]
where $S^{\rho}(h, q, p)=S^{\rho}_{(h)}(q, p)$. 

From (\ref{gener2}) we deduce that 
\begin{equation} \label{mnb}
p_k\frac{\partial S^{\rho}_{(h)}}{\partial p}(q_{k-1}, p_k)-S^{\rho}_{(h)}(q_{k-1}, p_k)
-\left(p_k\frac{\partial S_{(h)}}{\partial p}(q_{k-1}, p_k)-S_{(h)}(q_{k-1}, p_k)\right) = {\mathcal O}(h^{\rho+1}) \, .
\end{equation}

\begin{prop}
The Hamiltonian function  $H^{(\rho)}_h: T^*Q\rightarrow {\mathbb R}$ is regular, that is, the matrix
\[
\left( \frac{\partial^2 H^{(\rho)}_h}{\partial p_i\partial p_j}\right)
\]
is nondegenerate, for all $\rho\in\mathbb{N}$ and $h$ small enough.
\end{prop}
\begin{proof}
From the construction of $H_h^{(\rho)}$ it is possible to check that, in coordinates, 
$S_1(q, p)=H(q,p)$. 
Therefore, from (\ref{pois2}) we have 
\[
q_{k}=q_{k-1}+h\frac{\partial H}{\partial p}(q_{k-1}, p_k)+{\mathcal O}(h^2)
\]
and from Equations (\ref{pois1})  and (\ref{eqo1}) we get
\[
p_{k}+h\frac{\partial H}{\partial q}(q_{k-1}, p_k)=F_d^h(q_{k-1}, q_{k-1}+h\frac{\partial H}{\partial p}(q_{k-1}, p_k))+{\mathcal O}(h^2) \, .
\]
Taking derivatives of the last expression  with respect to $p_k$, we obtain
\[
I=h\frac{\partial F_d^h}{\partial q_k}\frac{\partial^2 H}{\partial p\partial p}+{\mathcal O}(h) \, ,
\]
where $I$ is the identity matrix. From it,  we deduce that 
\[
\frac{\partial F_d^h}{\partial q_k}={\mathcal O}(1/h) \quad \mbox{as} \quad h\rightarrow 0
\]
and the regularity of 
$
\displaystyle\left( \frac{\partial^2 H^{(\rho)}_h}{\partial p_i\partial p_j}\right)
$
follows.
\end{proof}
Since $H^{(\rho)}_h: T^*Q\rightarrow {\mathbb R}$ is regular for $h$ small enough, it is possible to define the regular Lagrangian functions
$L^{(\rho)}_h: TQ\rightarrow {\mathbb R}$ by
\[
L^{(\rho)}_h(v_q)=\langle \alpha_q, v_q\rangle -H^{(\rho)}_h(\alpha_q) \, ,
\]
where $v_q={\mathbb F}H(\alpha_q)$ and ${\mathbb F}H: T^*Q\rightarrow TQ$ is the mapping
defined by
\[
\langle {\mathbb F}H(\alpha_q), \beta_q\rangle=\frac{d}{dt}\Big|_{t=0}H(\alpha_q+t\beta_q) 
\]
for all $\alpha_q, \beta_q\in T^*_qQ$.

The relation between the Lagrangian $L^{(\rho)}_h$  and the  type 2 generating function  
$S^{\rho}_{(h)}$ is given by
\[  
p_k\frac{\partial S^{\rho}_{(h)}}{\partial p}(q_{k-1}, p_k)-S^{\rho}_{(h)}(q_{k-1}, p_k)=
\int^h_0 L^{(\rho)}_h(q_{\rho}(t), \dot{q}_{\rho}(t))\; dt \, ,
\]
where $q_{\rho}(t)$ is the unique solution of the Euler-Lagrange equations for $L^{(\rho)}_h$ satisfying the boundary conditions $q(0) = q_{k-1}$ and $q(h)=q_k$, where 
\begin{align*}
p_{k-1}&=p_{k}+\frac{\partial S^{\rho}_{(h)}}{\partial q}(q_{k-1}, p_k) \, ,\\
q_{k}&=q_{k-1}+\frac{\partial S^{\rho}_{(h)}}{\partial p}(q_{k-1}, p_k) \, .
\end{align*}
Using  Equation (\ref{mnb}) we obtain
\begin{align*}
L_d^h(q_{\rho}(0), q_{\rho}(h))&=
L_d^h\left(q_{\rho}(0), q_{\rho}(0)+\frac{\partial S^{\rho}_{(h)}}{\partial p}(q_{\rho}(0), p_{\rho}(h))\right)\\
&=L_d^h\left(q_{\rho}(0), q_{\rho}(0)+\frac{\partial S_{(h)}}{\partial p}(q_{\rho}(0), p_{\rho}(h))\right)+{\mathcal O}(h^{\rho+1})
\\
&=
p_{\rho}(h)\frac{\partial S_{(h)}}{\partial p}(q_{\rho}(0), p_{\rho}(h))-S_{(h)}(q_{k-1}, p_k)+{\mathcal O}(h^{\rho+1})\\
&\stackrel{\mathclap{(\ref{mnb})}}{=}
p_{\rho}(h)\frac{\partial S^{\rho}_{(h)}}{\partial p}(q_{\rho}(0), p_{\rho}(h))-S_{(h)}^{\rho}(q_{\rho}(0), p_{\rho}(h))+{\mathcal O}(h^{\rho+1})\\
&=\int_0^h L^{(\rho)}_h (q_{\rho}(t), \dot{q}_{\rho}(t))\; dt + {\mathcal O}(h^{\rho+1}) \, .
\end{align*}
Therefore, we have deduced that $L_d^h$ is a discrete Lagrangian of variational order $\rho$ for $L^{(\rho)}_h$ (see \cite{mats} for a related result).

\begin{ex}
Consider the  second order difference equation 
\begin{equation} \label{backward:example}
\frac{1}{h^2}(x_2-2x_1+x_0)+x_1=0, \quad h>0 \, , 
\end{equation}
which defines a submanifold $M_h$ of ${\mathbb R}^3$. $M_h$  corresponds to the smooth family of sections  
\[
\Gamma^h_d (x_{0}, x_{1})=(x_{0}, x_{1}, (2-h^2)x_1-x_0) \, .
\]
We can define a family of local diffeomorphisms $F^h_d: {\mathbb R}\times {\mathbb R}\rightarrow T^*{\mathbb R}$,
\[
F^h_d(x_0, x_1)=\left(x_0, \frac{x_1-x_0}{h}+hx_0+bh\right) \quad \mbox{with} \quad b\in {\mathbb R}\, ,
\]
which produce a family of Lagrangian submanifolds of $(T^*{\mathbb R}\times T^*{\mathbb R}, \Omega_{\mathbb R})$. The corresponding symplectomorphism on $T^*Q$ is given by
\[
\Psi_{F_d^h, M_h}(x_0, p_0)=((1-h^2)x_0+hp_0-bh^2, p_0-hx_0)\, .
\]
Observe that $\lim_{h\rightarrow 0^+} \Psi_{F_d^h, M_h}=\operatorname{Id}_{T^*{\mathbb R}}$.

The type 1 generating function for the Lagrangian submanifold $\hbox{graph}(\Psi_{F_d^h, M_h})$ is
\[
L_d(x_0, x_1)=\frac{h}{2}\left( \frac{x_1-x_0}{h}\right)^2-\frac{1}{2}hx_0^2+bh(x_1-x_0) \, ,
\] 
which has (\ref{backward:example}) as discrete Euler-Lagrange equations.
The corresponding type 2 generating function is
\[
S_{(h)}(x_0, p_1)=\frac{h}{2}(x_0^2+p_1^2)-h^2bp_1+\frac{b^2h^3}{2}\, .
\]
Using Theorem \ref{thm:hairer} we obtain a Hamiltonian function of the form 
\[
H_h(x, p)=\frac{1}{2}(x^2+p^2)-h(bp+\frac{1}{2}xp)+{\mathcal O}(h^2)\, .
\]
If we take the truncation
\[
H^{(2)}_h(x, p)=\frac{1}{2}(x^2+p^2)-h(bp+\frac{1}{2}xp)
\]
we obtain the associated Lagrangian 
\[
 L^{(2)}_h(x, \dot{x})= \frac{1}{8} \left(4 b^2 h^2+\left(-4+h^2\right) x^2+4 h x \dot{x}+4 \dot{x}^2+4 b h (h x+2 \dot{x})\right) \, .
\]
The corresponding exact discrete Lagrangian is
\begin{align*}
L_d^{(2)e}(x_0, x_1)&=\int_0^hL^{(2)}_h(x(t), \dot{x}(t))\; dt\\
&=
\frac{x_0{}^2-2 x_0 x_1+x_1{}^2}{2 h}+\frac{1}{12} \left(-12 b x_0-5 x_0{}^2+12 b x_1-2 x_0 x_1+x_1{}^2\right) h + {\mathcal O}(h^3) \, .
\end{align*}
Observe that 
\[
L_d^{(2)e}(x_0, x_1)-L_d(x_0, x_1)=\frac{1}{12} \left(x_0{}^2-2 x_0x_1+x_1{}^2\right) h+{\mathcal O}(h^3) \, ,
\]
and along the solutions $x(t)$ of the Euler-Lagrange equations for $L^{(2)}_h$ we get
\[
L_d^{(2)e}(x(0), x(h))-L_d(x(0), x(h))={\mathcal O}(h^3) \, .
\]
\end{ex}

\section{Alternative Lagrangian formulations} \label{sec5}

In this section we will first recall how a class of constants of motion arises from alternative Lagrangian formulations of a SODE, with the two alternative Lagrangians being genuinely different in the sense that they should not differ by a constant and/or addition of a total time derivative \cite{Crampin1983,MR1983}. Then we show that the same phenomenon occurs in the discrete setting.

\subsection{Continuous SODEs}
It is well known, see for instance \cite{Crampin1983,MR1983}, that given a vector field $\Gamma$ on a manifold $M$, if we can find a (1,1)-tensor field $A$ on $M$ such that $\mathcal{L}_\Gamma A=0$ then also $\mathcal{L}_\Gamma A^k=0$ and therefore $\operatorname{Tr}(A^k)$ is a constant of motion for $\Gamma$ for all $k$.

It is possible to construct such a (1,1)-tensor field if we have the following ingredients. Assume $(M,\omega)$ is a symplectic manifold, $\Gamma$ is a Hamiltonian vector field on $M$ with respect to $\omega$ and $\tilde{\omega}$ is a two-form on $M$ such that $\mathcal{L}_\Gamma\tilde{\omega}=0$. Then we can define $A$ from the condition 
\begin{equation} \label{def-A}
i_X\tilde{\omega}=i_{A(X)}\omega \quad \mbox{for all} \quad X\in \mathfrak{X}(M) \, ,
\end{equation}
that is, $A(X)=(\sharp_\omega\circ\flat_{\tilde{\omega}})(X)$.
\[
\xymatrix{
& \Omega^1(M) \ar[dr]^{\sharp_\omega} & \\
\mathfrak{X}(M) \ar[rr]_A \ar[ur]^{\flat_{\tilde{\omega}}} & & \mathfrak{X}(M)
}
\]
Then the conditions $\mathcal{L}_\Gamma\omega=\mathcal{L}_\Gamma\tilde{\omega}=0$ imply $\mathcal{L}_\Gamma A=0$. Indeed, taking Lie derivatives with respect to $\Gamma$ on both sides of (\ref{def-A}) we obtain
\[
i_{\left[\Gamma,X\right]}\tilde{\omega}+i_{X}\mathcal{L}_\Gamma\tilde{\omega} = i_{\left[\Gamma,A(X)\right]}\omega+i_{A(X)}\mathcal{L}_\Gamma \omega \quad \mbox{ for all } \quad X \in \mathfrak{X}(M) \, , 
\]
that is, $i_{\left[\Gamma,X\right]}\tilde{\omega}= i_{\left[\Gamma,A(X)\right]}\omega$ and therefore, again from (\ref{def-A}), we get $A\left[\Gamma,X\right]=\left[\Gamma,A(X)\right]$, that is, $\mathcal{L}_\Gamma A=0$.

The above situation arises for instance if we have two alternative Lagrangian formulations for $\Gamma$, with Lagrangian functions $L$ and $\tilde{L}$ (see \cite{MR1983}). Since we do not need to make any assumptions on the rank of the 2-form $\tilde{\omega}$, it is enough to require that one of the Lagrangians, say $L$, is regular. Then the corresponding Poincar\'{e}-Cartan two-forms $\omega_{L}$ and $\omega_{\tilde{L}}$ can be used to construct the (1,1)-tensor field $A=\sharp_{\omega_L}\circ\flat_{\omega_{\tilde{L}}}$, which satisfies $\mathcal{L}_\Gamma A=0$ since $\Gamma$ is Hamiltonian with respect to both $\omega_{L}$ and $\omega_{\tilde{L}}$ and therefore $\mathcal{L}_\Gamma\omega_L=\mathcal{L}_\Gamma\omega_{\tilde{L}}=0$.

\subsection{Discrete SOdEs}
Assume there are two alternative regular discrete Lagrangians $L_d$ and $\tilde{L}_d$ for a discrete second order difference equation $\Gamma$ on $Q\times Q$. Then we get two discrete Lagrangian symplectic forms $\Omega_{L_d}$ and $\Omega_{\tilde{L}_d}$ \cite{marsden-west} (equivalently, if we can find $F$ and $\tilde{F}$ then from Propositon \ref{discrete-crampin} we obtain $\Omega_d$ and $\tilde{\Omega}_d$). We can define a (1,1)-tensor field $A_d$ on $Q\times Q$ as before, from the condition 
\begin{equation} \label{def-Ad}
i_X\tilde{\Omega}_d=i_{A_d(X)}\Omega_d \quad  \mbox{for all} \quad X\in \mathfrak{X}(Q\times Q) \, .
\end{equation}
\[
\xymatrix{
& \Omega^1(Q\times Q) \ar[dr]^{\sharp_{\Omega_d}} & \\
\mathfrak{X}(Q\times Q) \ar[rr]_{A_d} \ar[ur]^{\flat_{\tilde{\Omega}_d}} & & \mathfrak{X}(Q\times Q)
}
\]
Notice again that only the regularity of $L_d$ is actually needed. 
We define the discrete Lie derivative of $A_d$ along $\Gamma$ as 
\[
\mathcal{L}^d_\Gamma A_d=\Phi_\Gamma^{*}\circ A_d - A_d\circ \Phi_\Gamma^{*} \, ,
\]
where $\Phi_\Gamma^{*}=(\Phi_\Gamma^{-1})_{*}:\mathfrak{X}(Q\times Q)\longrightarrow \mathfrak{X}(Q\times Q)$, following the simplified notation used in \cite{foundation}. The conditions $\mathcal{L}^d_\Gamma\Omega_d=\mathcal{L}^d_\Gamma\tilde{\Omega}_d=0$ imply, as in the continuous case, that $\mathcal{L}^d_\Gamma A_d=0$. Indeed, if we take discrete Lie derivatives with respect to $\Gamma$ on both sides of (\ref{def-Ad}), we obtain
\[
i_{\Phi^{*}_\Gamma X}\Phi^{*}_\Gamma\tilde{\Omega}_d=\Phi^{*}_\Gamma(i_{X}\tilde{\Omega}_d)=\Phi^{*}_\Gamma(i_{A_d X}\Omega_d)=i_{\Phi^{*}_\Gamma A_d X}\Phi^{*}_\Gamma\Omega_d \, ,
\]
which using $\mathcal{L}^d_\Gamma\Omega_d=\mathcal{L}^d_\Gamma\tilde{\Omega}_d=0$ becomes $i_{\Phi^{*}_\Gamma X}\tilde{\Omega}_d=i_{\Phi^{*}_\Gamma A_d X}\Omega_d$. Then, by definition of $A_d$, we obtain that $A_d \Phi_\Gamma^{*}X=\Phi_\Gamma^{*}A_d X$. Observe that the condition $A_d\Phi_\Gamma^{*}X=\Phi_\Gamma^{*}A_d X$ is equivalent to $A_d(\Phi_\Gamma)_{*}X=(\Phi_\Gamma)_{*}A_d X$.

Choose a basis $\left\{X_1,\ldots,X_{2n} \right\}$ of $\mathfrak{X}(Q\times Q)$ and write $A_d(X_a)=\mathcal{A}_a^bX_b$, $(\Phi_\Gamma)_{*}(X_a)=\phi_a^bX_b$. Then the above condition takes the form
\begin{align*}
0&=(\Phi_\Gamma)_{*}\circ A_d (X_a)-A_d\circ (\Phi_\Gamma)_{*} (X_a) \\
&=(\Phi_\Gamma)_{*}(\mathcal{A}_a^b(x)X_b(x))-A_d(\phi_a^b(x)X_b(\Phi_\Gamma(x))) \\
&=\mathcal{A}_a^b(x)\phi_b^c(x)X_c(\Phi_\Gamma(x))-\phi_a^b(x)\mathcal{A}_b^c(\Phi_\Gamma(x))X_c(\Phi_\Gamma(x)) \, ,
\end{align*}
from which we get $\mathcal{A}_a^b(x)\phi_b^c(x)=\phi_a^b(x)\mathcal{A}_b^c(\Phi_\Gamma(x))$, that is, $\mathcal{A}_d^c(\Phi_\Gamma(x))=(\phi^{-1})^a_d(x)\mathcal{A}_a^b(x)\phi_b^c(x)$. 
Therefore the eigenvalues of $A_d(x)$ and $A_d(\Phi_\Gamma(x))$ coincide and in particular $\operatorname{Tr}A_d^k(x)=\operatorname{Tr}A_d^k(\Phi_\Gamma(x))$, that is, $\operatorname{Tr}A_d^k$ is a constant of motion for $\Gamma$.

\begin{ex}
Consider the second order differential equation $\ddot{x}+x=0$ on $\mathbb{R}$, that is, the SODE $\Gamma=\dot{x}\frac{\partial}{\partial x}-x\frac{\partial}{\partial \dot{x}}\in \mathfrak{X}(\mathbb{R}^2)$. 
We will find a discretization  of the system, which admits two alternative discrete Lagrangians $L_{d1}$ and $L_{d2}$, and for which $A_d=\sharp_{\Omega_{L_{d1}}}\circ \flat_{\Omega_{L_{d2}}}$ provides constants of motion.
The solutions to the continuous system are given by $x(t)=a\cos(t)+b\sin(t)$, where $a$ and $b$ are constants. Therefore the exponential map associated with $\Gamma$ is given by
\[
\begin{array}{cccl}
\operatorname{exp}^\Gamma_{(x,h)}: & TQ & \longrightarrow & Q\times Q \\
& (x,v) & \longmapsto & (x,x\cos(h)+v\sin(h))
\end{array}
\]
and the flow at time $h$ is
\[
\begin{array}{cccl}
\Phi^\Gamma_h : & TQ & \longrightarrow & TQ \\
& (x,v) & \longmapsto & (x\cos(h)+v\sin(h),v\cos(h)-x\sin(h)) \, .
\end{array}
\]
Notice that for the continuous system we have the 
two alternative Lagrangians
\[
L=\frac{1}{2}(\dot{x}^2-x^2) \quad \mbox{and} \quad \tilde{L}=\frac{1}{3}\dot{x}^4+2x^2\dot{x}^2-x^4 \, ,
\]
with corresponding Legendre transformations $F_1(x,\dot{x})=(x,\dot{x})$ and $F_2(x,\dot{x})=(x,\frac{4}{3}\dot{x}^3+4x^2\dot{x})$.
Therefore $\operatorname{Im}(F_1\times F_1)\circ(\operatorname{Id}\times \Phi_t^{\Gamma})$ and $\operatorname{Im}(F_2\times F_2)\circ(\operatorname{Id}\times \Phi_t^{\Gamma})$ are both Lagrangian submanifolds of $(T^*Q\times T^*Q,\Omega_Q)$.
Then we can define the discrete SOdE
\[
\begin{array}{cccc}
\Gamma_d= (\operatorname{exp}^\Gamma_h \times \operatorname{exp}^\Gamma_h) \circ (\operatorname{Id}\times \Phi^\Gamma_h) \circ R^{e^-}_{h}: & Q\times Q  & \longrightarrow & Q\times Q \times Q\times Q \\
& (x_0,x_1) & \longmapsto  &  (x_0,x_1,x_1,2x_1\cos(h)-x_0)
\end{array}
\]
and the discrete Legendre transformations $F_{d1}=F_1\circ R^{e^-}_{h}$ and $F_{d2}=F_2\circ R^{e^-}_{h}$ which, according to the  diagram in Theorem \ref{thm-cont-disc}, provide Lagrangian submanifolds $\operatorname{Im}(F_{d1} \times F_{d1})\circ\Gamma_d$ and $\operatorname{Im}(F_{d2} \times F_{d2})\circ \Gamma_d$ (and corresponding discrete Lagrangians $L_{d1}$ and $L_{d2}$).

The Lagrangian submanifolds $\operatorname{Im}(F_{d1} \times F_{d1})\circ\Gamma_d$ and $\operatorname{Im}(F_{d2} \times F_{d2})\circ\Gamma_d$ are given respectively by
\[
\left( x_0, \frac{x_1-x_0\cos(h)}{\sin(h)}, x_1, \frac{x_1\cos(h)-x_0}{\sin(h)} \right) \quad \mbox{and}
\]
\[
\left( x_0, \frac{4}{3}\left(\frac{x_1-x_0\cos(h)}{\sin(h)}\right)^3+4x_0^2\left(\frac{x_1-x_0\cos(h)}{\sin(h)}\right), x_1, \frac{4}{3}\left(\frac{x_1\cos(h)-x_0}{\sin(h)}\right)^3+4x_1^2\left(\frac{x_1\cos(h)-x_0}{\sin(h)}\right) \right) \, ,
\]
from where we get
\[
\Omega_{L_{d1}}=\frac{-1}{\sin(h)}dx_0\wedge dx_1 \quad \mbox{and} \quad
\Omega_{L_{d2}}= -4\left( \frac{x_1^2-2x_0x_1\cos(h)+x_0^2}{\sin^3(h)} \right)dx_0 \wedge dx_1 \, .
\]
Therefore we have
\[
A_d=\frac{4}{\sin^2(h)}(x_1^2-2x_0x_1\cos(h)+x_0^2)dx_0\otimes\frac{\partial}{\partial x_0}+\frac{4}{\sin^2(h)}(x_1^2-2x_0x_1\cos(h)+x_0^2)dx_1\otimes\frac{\partial}{\partial x_1}
\]
and we obtain the conserved quantity $x_1^2-2x_0x_1\cos(h)+x_0^2$ for the SOdE  $\Gamma_d$, which is a discretization of the conserved quantity $\dot{x}^2+x^2$ for $\Gamma$.

Although they are not needed in order to get $A_d$, the two discrete Lagrangians that we obtain are
\begin{align*}
L_{d1}(x_0,x_1)&= \frac{\cos(h)}{2\sin(h)}\left(x_0^2+x_1^2\right) -\frac{x_0x_1}{\sin(h)} \, ,\\
L_{d2}(x_0,x_1)&= x_1^4 \cot(h) -\frac{4}{3}x_0 x_1^3 \csc(h) +\frac{1}{3}x_0^4 \cos(2h)\csc(h)\sec(h) + \frac{1}{3}(x_1\cot(h)-x_0\csc(h))^4\tan(h) \\
&= \cot(h)\left( 1+\frac{\cot^2(h)}{3} \right)x_1^4 - \frac{4}{3}x_0 x_1^3 \csc^3(h) + 2 x_0^2 x_1^2 \cot(h)\csc^2(h) - \frac{4}{3} x_0^3 x_1 \csc^3(h) \\
&\mathrel{\phantom{=}}  + \cot(h)\left( 1+\frac{\cot^2(h)}{3} \right)x_0^4 \, .
\end{align*}

\end{ex}

\section{Variationality of discrete constrained systems} \label{sec6}

Now we will consider the case of constrained second order discrete systems. 
Let $M_d \subset Q\times Q$ be a submanifold defined by the discrete constraints $q^\alpha_{k}=\psi^\alpha_{k}(q^i_{k-1},q^a_{k})$, where $a,b=1,\ldots,m<n$, $\alpha,\beta=m+1,\ldots,n$, and let $\Gamma_{d}$ be an explicit second order difference equation on $M_d$, that is, $\Gamma_d$ is a map 
\[
\Gamma_d: M_d  \longrightarrow M_d\times M_d
\]
satisfying $\alpha_{M_d}\circ\Gamma_d=\operatorname{Id}$ and $\operatorname{Im}(i\times i)\circ\Gamma_d\subset \ddot{Q}_d$, where $\ddot{Q}_d$ is the discrete second order submanifold defined in Section \ref{explicit}, $i:M_d\hookrightarrow Q\times Q$ denotes the inclusion and $\alpha_{M_d}:M_d\times M_d \longrightarrow M_d$ denotes the projection onto the first factor. Locally $\Gamma_d$ is given by 
\[
\Gamma_d(q^i_{k-1},q^a_{k}) = (q^i_{k-1},q^a_{k},q^a_{k},\psi^\alpha_{k},q^a_{k+1}=\Gamma^a (q^j_{k-1},q^b_{k})) \, .
\]
We will also use the notation $q_{\hat{k}}=(q_k^a)$ for $a,b=1,\ldots,m<n$ and $q_{\bar{k}}=(q_k^\alpha)$ for $\alpha,\beta=m+1,\ldots,n$. 

Given an immersion $F:M_d\longrightarrow T^{*}Q$ we define $\gamma_{F,\Gamma}:=(F\times F)\circ \Gamma_d$, as shown in the following commutative diagram:
\[
\xymatrix{
M_d \times M_d \ar[rr]^{F\times F} & & T^{*}Q \times T^{*}Q \ar[d]^{\alpha_{T^*Q}} \\
M_d \ar[u]^{\Gamma_d} \ar[rr]^{F} \ar[urr]^{\gamma_{F,\Gamma_d}} \ar[dr]_{\alpha_{Q_{|M_{d}}}} & & T^{*}Q \ar[dl]^{\pi_{Q}} \\
& Q &
}
\]

\begin{definition} \label{variational-immersion}
A SOdE $\Gamma_{d}$ on  $M_d$ is variational if there exists an immersion $F:M_d \longrightarrow T^{*}Q$ such that $\operatorname{Im}(\gamma_{F,\Gamma_d})$ is an isotropic submanifold of $(T^{*}Q\times T^{*}Q,\Omega_Q)$.
\end{definition}

The above diagram in local coordinates becomes
\[
\xymatrix{
(q_{k-1}^i,q_k^a,q_k^a,\psi_k^\alpha,\Gamma^a(q_{k-1}^j,q_k^b)) \ar[rr]^/-20pt/{F\times F} & & (q_{k-1}^i,F_i(q_{k-1}^j,q_k^b),q_k^a,\psi_k^\alpha,F_i(q_k^b,\psi_k^\beta,\Gamma^b(q_{k-1}^j,q_k^b))) \ar[d]^{\alpha_{T^*Q}} \\
(q_{k-1}^i,q_k^a) \ar[u]^{\Gamma_{d}} \ar[rr]^{F} \ar[urr]^{\gamma_{F,\Gamma_d}}  & & (q_{k-1}^i,F_i(q_{k-1}^j,q_k^b)) 
}
\]
Then the condition
\[
d \left( F_a(q_k^b,\psi_k^\beta,\Gamma^b(q_{k-1},q_{\hat{k}}))dq_k^a+F_\alpha(q_k^b,\psi_k^\beta,\Gamma^b(q_{k-1},q_{\hat{k}}))d\psi_k^\alpha -F_i(q_{k-1}^j,q_k^b)dq_{k-1}^i \right)=0
\]
gives the discrete Helmholtz conditions for constrained systems.

First we will provide an extension of Theorem \ref{discrete-crampin} to the discrete setting with constraints. We will need the following proposition from \cite{GuiStern}.

\begin{prop}[\cite{GuiStern}]\label{Prop:TildeS}
	Let $f:M \longrightarrow N$ be an immersion. For each Lagrangian submanifold $S\subset T^{*}M$ we can define a Lagrangian submanifold $\tilde{S}\subset T^{*}N$ by 
	\[
	\tilde{S}=\left\{ \mu\in T^{*}N: f^{*}\mu\in S \right\}.
	\]
\end{prop}

Denote the flow of an explicit constrained second order difference equation $\Gamma_d:M_d\longrightarrow M_d\times M_d$ by $\Phi_{\Gamma_d}:M_d\longrightarrow M_d$, such that $\Phi_{\Gamma_d}(q_{k-1},q_{\hat{k}}) = (q_{\hat{k}},\psi_{\bar{k}},\Gamma_{\widehat{k+1}})$.

\begin{prop} 
	An explicit constrained second order difference equation $\Gamma_d:M_d\longrightarrow M_d\times M_d$ is variational if and only if there is a nondegenerate two-form $\Omega_d$ on $M_d$ such that
	\begin{enumerate}
		\item $\mathcal{L}^{d}_{\Gamma_d}\Omega_d=0 \, ,$
		\item $\Omega_d(V_1,V_2)=0$ for all $V_1,V_2\in \operatorname{Ker}(T\alpha_{Q_{|M_{d}}}) \, ,$
		\item $d\Omega_d=0 \, ,$
		\item $\left.\flat_{\Omega_d}\right|_{\operatorname{Ker}(T\alpha_{Q_{|M_{d}}})}$ is injective,
	\end{enumerate}
	where $\mathcal{L}^{d}_{\Gamma_d}\Omega_d:=(\Phi_{\Gamma_d})^{*}\Omega_d-\Omega_d$ is regarded as a discrete analogue of the Lie derivative.
\end{prop}

\begin{proof}
	The proof goes along the same lines as the analogue in \cite{BFM}.
	
	If we assume that $\Gamma_d$ is variational, we can define $\Omega_d=d(F^{*}\theta_Q)$ which clearly satisfies condition (iii). From the local expression 
	\[
	\Omega_d=\frac{\partial F_i}{\partial q_{k-1}^j}dq^j_{k-1}\wedge dq_{k-1}^i+\frac{\partial F_i}{\partial q_k^b}dq_k^b\wedge dq_{k-1}^i
	\]
	condition (ii) is also clear since $\operatorname{Ker}(T\alpha_{Q_{|M_{d}}})=\mbox{span}\left\{\frac{\partial}{\partial q_k^b}\right\}$. The requirement of $F$ being  an immersion implies that $\left(\frac{\partial F_i}{\partial q_k^b}\right)$ is of maximal rank. Thus, taking $v_1,v_2\in \operatorname{Ker}(T\alpha_{Q_{|M_{d}}})$, $v_1=v_1^b\frac{\partial}{\partial q_k^b}$, $v_2=v_2^b\frac{\partial}{\partial q_k^b}$ such that $i_{v_{1}}\Omega_d-i_{v_{2}}\Omega_d=(v_1^b-v_2^b)\left(\frac{\partial F_i}{\partial q_k^b}\right)dq_{k-1}^i=0$, we obtain $v_1=v_2$ because of the rank condition.  Therefore condition (iv) is satisfied.
	
Notice that
\[
\mathcal{L}_{\Gamma_d}^d \Omega_d=\Phi_{\Gamma_d}^{*}\Omega_d-\Omega_d=d\Phi_{\Gamma_d}^{*} F^* \theta_Q-dF^* \theta_Q=d(\mathcal{L}_{\Gamma_d}^d F^* \theta_Q) \, .
\] 
In order to check condition (i) we locally compute $\mathcal{L}_{\Gamma_d}^d F^* \theta_Q=(F\circ\Phi_{\Gamma_d})^{*}\theta_Q-F^{*}\theta_Q$  to get
\[
\mathcal{L}^{d}_{\Gamma_d}F^*\theta_Q=
F_{a}(q_{\hat{k}},\psi_{\bar{k}},\Gamma_{\hat{k}})dq_{k}^a+F_\alpha(q_{\hat{k}},\psi_{\bar{k}},\Gamma_{\hat{k}})d\psi_{k}^a-F_i(q_{k-1},q_{\hat{k}})dq_{k-1}^i \, ,
\]
since $(F\circ\Phi_{\Gamma_d})(q_{k-1},q_{\hat{k}})=(q_{\hat{k}},\psi_{\bar{k}},F_i(q_{\hat{k}},\psi_{\bar{k}},\Gamma_{\hat{k}}))$.
Note that the condition $d(\mathcal{L}_{\Gamma_d}^d F^* \theta_Q)=0$ is exactly the same as requiring that $\operatorname{Im}(\gamma_{F,\Gamma_d})$ be isotropic.

	Conversely, let $\Omega_d$ be a two-form on $M_d$ satisfying (i)-(iv). From (iii), locally $\Omega_d=d\Theta$ for a one-form $\Theta$ on $M_d$ and from (ii) $\Theta$ has the local expression
	\[
	\Theta=\alpha_i dq_{k-1}^i+\frac{\partial h}{\partial q_k^b}(q_{k-1},q_{\hat{k}})dq_k^b
	\]
	for a locally defined map $h:M_d\longrightarrow \mathbb{R}$. Define $\bar{\Theta}=\Theta-dh$, which satisfies $\bar{\Theta}(V)=0$ for all $V\in \operatorname{Ker}(T\alpha_{Q_{|M_{d}}})$ and $d\bar{\Theta}=\Omega$. Then  $F:M_d\longrightarrow T^*Q$ is given by
	\[
	\langle F(q_{k-1},q_{\hat{k}}), v_{q_{k-1}} \rangle=\langle \bar{\Theta}(q_{k-1},q_{\hat{k}}), V_{v_{q_{k-1}}} \rangle \mbox{ for all }  v_{q_{k-1}}\in TQ,
	\]
	where $V_{v_{q_{k-1}}}$ in $TM_d$ is any vector satisfying $T\alpha_{Q_{|M_{d}}}(V_{v_{q_{k-1}}})=v_{q_{k-1}}$.
	
	Since the one-form $\mathcal{L}_{\Gamma_d}^d\bar{\Theta}=\mathcal{L}_{\Gamma_d}^d F^* \theta_Q$ is closed, we obtain a Lagrangian submanifold $\operatorname{Im}(\mathcal{L}_{\Gamma_d}^d F^* \theta_Q)$ of $(T^*M_d,\omega_{M_d})$. Using Proposition \ref{Prop:TildeS} (with $N=Q\times Q$) we obtain a Lagrangian submanifold of $(T^*(Q\times Q),\omega_{Q\times Q})$, described by
	\[
	\widetilde{\operatorname{Im}(\mathcal{L}_{\Gamma_d}^d F^* \theta_Q)}=\left\{ \mu\in T^*(Q\times Q) : i_M^*\mu \in \operatorname{Im}(\mathcal{L}_{\Gamma_d}^d F^* \theta_Q) \right\} \, ,
	\]
	where $i_M$ denotes the inclusion. In coordinates,  $\widetilde{\operatorname{Im}(\mathcal{L}_{\Gamma_d}^d F^* \theta_Q)}$ is given by
	\[
	\left( q_{k-1}^i,q_{\hat{k}}, \psi_{\bar{k}}, -F_i+F_\alpha\frac{\partial \psi^\alpha_{k}}{\partial q_{k-1}^i}-p_\alpha \frac{\partial \psi^\alpha_{k}}{\partial q_{k-1}^i}, F_a+F_\alpha\frac{\partial \psi^\alpha_{k}}{\partial q_{k}^a}-p_\alpha\frac{\partial \psi^\alpha_{k}}{\partial q_{k}^a},p_\alpha \right) \, .
	\]
	Since $\operatorname{Im}(\Psi^{-1}\circ\gamma_{F,\Gamma_d})\subset \widetilde{\operatorname{Im}(\mathcal{L}_{\Gamma_d}^d F^* \theta_Q)}$, $\operatorname{Im}(\gamma_{F,\Gamma_d})$ is an isotropic submanifold of $(T^*Q\times T^*Q,\Omega_Q)$. Furthermore, condition (iv) implies that $\left( \frac{\partial F_i}{\partial q_k^b} \right)$ has maximal rank, that is, $F$ is an immersion.
\end{proof}

Some natural questions that immediately arise are the following:
\begin{enumerate}
\item Given a continuous variational SODE $\Gamma$ on a submanifold $M\subset TQ$, find integrators $\Gamma_{d}$ that are also variational in the sense of Definition \ref{variational-immersion}.
\item From the existing integrators for nonholonomic systems \cite{CM2001,GNI,perlmutter06,BZ}, detect the ones that preserve the variational property (see also \cite{BFMestdag}).
\end{enumerate}

One of the integrators mentioned in (ii) is the discrete Lagrange-d'Alembert (DLA) algorithm, derived from the so-called discrete Lagrange-d'Alembert principle \cite{CM2001}. Given a nonholonomic system, that is, a Lagrangian $L:TQ\rightarrow \mathbb{R}$ and a nonintegrable distribution $D\subset TQ$, it is necessary to choose a discrete Lagrangian $L_d$ and a discrete constraint space $D_q\subset Q\times Q$, satisfying $\mbox{diag}(Q\times Q)\subset D_d$ and $\mbox{dim}(D_d)=\mbox{dim}(D)$, and defined by the annihilation of functions $w_d^a:Q\times Q\rightarrow \mathbb{R}$, $a=1\ldots,m$, regarded as discretizations of the constraint one-forms. As explained in \cite{CM2001}, these discretizations should be chosen in a consistent way in order to get `a desired order of accuracy'.

The DLA integrator is then given by
\begin{align}
D_{1}L_{d}(q_{k},q_{k+1})+D_{2}L_{d}(q_{k-1},q_{k})&=\lambda_{a}w^{a}(q_{k}) \, , \label{eqn:EL} \\ 
w^{a}_{d}(q_{k},q_{k+1})&=0 \, , \label{eqn:DLA2}
\end{align}
where $\lambda_{a}$ are Lagrange multipliers, $w^{a}$ are the constraint one-forms, $L_d$ is a discrete Lagrangian and $w^{a}_{d}$ is a discretization of the contraint one-forms.

Next  we will study different choices of constraints and immersions $F:M_d\longrightarrow T^*Q$ for the example of the vertical rolling disk.

\begin{ex}[Vertical rolling disk] \label{disk}
The system represents a vertical disk rolling on a plane without sliding. It is defined on the configuration space $Q=S^1\times S^1\times\mathbb{R}^{2}$, with  coordinates $(\theta,\varphi,x,y)$, where $\theta$ denotes the angle of self-rotation, $\varphi$ the angle between the direction in which the disk moves and the $x$-axis and $(x,y)$ are the coordinates of the contact point. The kinetic Lagrangian is given by $L=\frac{1}{2}(\dot{\theta}^{2}+\dot{\varphi}^{2}+\dot{x}^{2}+\dot{y}^{2})$, where all parameters are set to one, and the nonholonomic constraints of rolling without sliding are  $\dot{x}=\cos(\varphi)\dot{\theta}$, $\dot{y}=\sin(\varphi)\dot{\theta}$, which define a submanifold $M\subset TQ$. Therefore the constraint one-forms are $w^{1}=dx-\cos(\varphi)d\theta$ and $w^{2}=dy-\sin(\varphi)d\theta$.

Recall that the immersion
\[
\begin{array}{lccc}
F_1: & M & \longrightarrow & T^{*}Q \\
& (\theta,\varphi,x,y,\dot{\theta},\dot{\varphi}) & \longmapsto & (\theta,\varphi,x,y,2\dot{\theta},\dot{\varphi},0,0)
\end{array}
\]
provides an isotropic submanifold $\operatorname{Im}(TF_1\circ\Gamma)$ of $TT^*Q$, and implies that $\Gamma$ is variational in the sense of \cite[Definition 5.1]{BFM}. An alternative immersion is given by 
\[
\begin{array}{lccc}
F_2: & M & \longrightarrow & T^{*}Q \\
& (\theta, \varphi, x, y, \dot{\theta},\dot{\varphi}) & \longmapsto & \left(\theta,\varphi,x,y,\frac{\dot{\theta}}{\dot{\varphi}}, \dot{\varphi}-\frac{\dot{\theta}^{2}}{2\dot{\varphi}^{2}}\left( 1+\cos(\varphi)+\sin(\varphi) \right),\frac{\dot{\theta}}{\dot{\varphi}},\frac{\dot{\theta}}{\dot{\varphi}}\right) \, ,
\end{array}
\]
which also provides an isotropic submanifold $\operatorname{Im}(TF_2\circ\Gamma)$ of $TT^*Q$, see \cite[Example 5.8]{BFM}. 

Now in order to derive a DLA integrator we can choose for instance the discretizations
\begin{align*}
L^{\frac{1}{2}}_{d}(q_{k},q_{k+1})&=\frac{1}{2}\left( \left(\frac{\theta_{k+1}-\theta_{k}}{h}\right)^{2}+\left(\frac{\varphi_{k+1}-\varphi_{k}}{h}\right)^{2}+\left(\frac{x_{k+1}-x_{k}}{h}\right)^{2}+\left(\frac{y_{k+1}-y_{k}}{h}\right)^{2} \right) \, , \\
w^{1}_{d}(q_{k},q_{k+1})&=\frac{x_{k+1}-x_{k}}{h}-\frac{\theta_{k+1}-\theta_{k}}{h}\cos\left(\frac{\varphi_{k}+\varphi_{k+1}}{2}\right) \, , \\
w^{2}_{d}(q_{k},q_{k+1})&=\frac{y_{k+1}-y_{k}}{h}-\frac{\theta_{k+1}-\theta_{k}}{h}\sin\left(\frac{\varphi_{k}+\varphi_{k+1}}{2}\right) \, .
\end{align*}
Equations (\ref{eqn:EL}) are then
\begin{align}
-\frac{\theta_{k+1}-\theta_{k}}{h^2}+\frac{\theta_{k}-\theta_{k-1}}{h^2}&=-\lambda_{1}\cos(\varphi_{k})-\lambda_{2}\sin(\varphi_{k})\label{eqn:theta} \, , \\
-\frac{\varphi_{k+1}-\varphi_{k}}{h^2}+\frac{\varphi_{k}-\varphi_{k-1}}{h^2}&=0 \, ,  \nonumber \\
-\frac{x_{k+1}-x_{k}}{h^2}+\frac{x_{k}-x_{k-1}}{h^2}&=\lambda_{1} \, , \nonumber \\ 
-\frac{y_{k+1}-y_{k}}{h^2}+\frac{y_{k}-y_{k-1}}{h^2}&=\lambda_{2} \, , \nonumber
\end{align}
from which we immediately obtain $\varphi_{k+1}=2\varphi_{k}-\varphi_{k-1}$.

The discrete constraints chosen above yield the Lagrange multipliers
\begin{align*}
\lambda_{1}&=-\frac{\theta_{k+1}-\theta_{k}}{h^2}\cos\left(\frac{\varphi_{k}+\varphi_{k+1}}{2}\right)+\frac{\theta_{k}-\theta_{k-1}}{h^2}\cos\left(\frac{\varphi_{k-1}+\varphi_{k}}{2}\right) \, , \\
\lambda_{2}&=-\frac{\theta_{k+1}-\theta_{k}}{h^2}\sin\left(\frac{\varphi_{k}+\varphi_{k+1}}{2}\right)+\frac{\theta_{k}-\theta_{k-1}}{h^2}\sin\left(\frac{\varphi_{k-1}+\varphi_{k}}{2}\right) \, ,
\end{align*}
and the substitution of them into (\ref{eqn:theta}) gives $\theta_{k+1}=2\theta_{k}-\theta_{k-1}$ (as long as $\varphi_k-\varphi_{k-1}\not=2(2n+1)\pi$, $n\in\mathbb{Z}$).

Therefore we have seen that $\Gamma_d$ is given by
\begin{multline*}
\Gamma_d(\theta_{k-1},\varphi_{k-1},x_{k-1},y_{k-1},\theta_{k},\varphi_{k}) = 
\bigg(\theta_{k-1}, \varphi_{k-1}, x_{k-1}, y_{k-1}, \theta_{k}, \varphi_{k}, \theta_{k}, \varphi_{k}, \\
x_{k-1}+\cos\Big(\frac{\varphi_{k-1}+\varphi_{k}}{2}\Big)(\theta_{k}-\theta_{k-1}),  
y_{k-1}+\sin\Big(\frac{\varphi_{k-1}+\varphi_{k}}{2}\Big)(\theta_{k}-\theta_{k-1}), 2\theta_{k}-\theta_{k-1}, 2\varphi_{k}-\varphi_{k-1} \bigg) \, .
\end{multline*}
If we define $F_{d1}:M_d \longrightarrow T^*Q$ in coordinates  by
\[
F_{d1}(\theta_{k-1},\varphi_{k-1},x_{k-1},y_{k-1},\theta_{k},\varphi_{k}) = \left(\theta_{k-1},\varphi_{k-1},x_{k-1},y_{k-1},2\frac{\theta_{k}-\theta_{k-1}}{h},\frac{\varphi_{k}-\varphi_{k-1}}{h},0,0\right) \, , 
\]
which is a discretization of $F_1$ given above, then $\operatorname{Im}((F_{d1}\times F_{d1})\circ\Gamma_d)$=$\operatorname{Im} (\gamma_{F_{d1},\Gamma_d})$ becomes
\begin{multline*}
  \bigg( \theta_{k-1}, \varphi_{k-1}, x_{k-1}, y_{k-1}, 2\frac{\theta_{k}-\theta_{k-1}}{h}, \frac{\varphi_{k}-\varphi_{k-1}}{h}, 0, 0, 
  \theta_{k},\varphi_{k}, \\
  x_{k-1}+\cos\Big(\frac{\varphi_{k-1}+\varphi_{k}}{2}\Big)(\theta_{k}-\theta_{k-1}), y_{k-1}+\sin\Big(\frac{\varphi_{k-1}+\varphi_{k}}{2}\Big)(\theta_{k}-\theta_{k-1}),
2\frac{\theta_{k}-\theta_{k-1}}{h}, \frac{\varphi_{k}-\varphi_{k-1}}{h},0,0
\bigg) \, .
\end{multline*}
Let $i:\operatorname{Im} (\gamma_{F_{d1},\Gamma_d}) \hookrightarrow T^{*}Q\times T^{*}Q$ denote the inclusion. Then $\operatorname{Im} (\gamma_{F_{d1},\Gamma_d})$ is an isotropic submanifold because
\begin{align*}
i^{*}\Omega_{Q}&=2 d\left(\frac{\theta_{k}-\theta_{k-1}}{h}\right)\wedge d\theta_{k}+d\left(\frac{\varphi_{k}-\varphi_{k-1}}{h}\right)\wedge d\varphi_{k}\\
&\mathrel{\phantom{=}}-2d\left(\frac{\theta_{k}-\theta_{k-1}}{h}\right)\wedge d\theta_{k-1}- d\left(\frac{\varphi_{k}-\varphi_{k-1}}{h}\right)\wedge  d\varphi_{k-1}=0 \, .
\end{align*}

For the chosen discrete Lagrangian $L_d^{\frac{1}{2}}$ and $F_{d1}$, but with arbitrary constraints, the isotropy condition is equivalent to
\begin{equation*}
2 d\left(\frac{\theta_{k+1}-\theta_{k-1}}{h}\right)\wedge d\theta_{k}-2 d\left(\frac{\theta_{k}-\theta_{k-1}}{h}\right)\wedge d\theta_{k-1}=0
\end{equation*}
since the choice of constraints does not affect the evolution of $\varphi$, given by $\varphi_{k+1}=2\varphi_{k}-\varphi_{k-1}$. Then we must  necessarily have an evolution of the form $\theta_{k+1}=-\theta_{k-1}+f(\theta_k)$ in order to obtain an isotropic  submanifold.
For instance, if we choose the alternative constraints
\begin{align*}
w^{1}_{d}(q_{k},q_{k+1})&=\frac{x_{k+1}-x_{k}}{h}-\frac{\theta_{k+1}-\theta_{k}}{h}\cos\left(\varphi_{k}\right) \, , \\
w^{2}_{d}(q_{k},q_{k+1})&=\frac{y_{k+1}-y_{k}}{h}-\frac{\theta_{k+1}-\theta_{k}}{h}\sin\left(\varphi_{k}\right) \, ,
\end{align*}
then we get the evolution $\theta_{k+1}=\theta_{k}+\left( \frac{\theta_{k}-\theta_{k-1}}{2} \right)(1+\cos(\varphi_{k}-\varphi_{k-1}))$ and therefore $\operatorname{Im}(F_{d1}\times F_{d1})\circ \Gamma_d$ is not an isotropic submanifold.

On the other hand, if we take the discrete constraints 
\begin{align*}
w^{1}_{d}(q_{k},q_{k+1})&=\frac{x_{k+1}-x_{k}}{h}-\frac{\theta_{k+1}-\theta_{k}}{h}\left( \frac{1}{2}\cos\left((1-\alpha)\varphi_{k}+\alpha\varphi_{k+1} \right) + \frac{1}{2}\cos\left(\alpha\varphi_{k}+(1-\alpha)\varphi_{k+1} \right)  \right) \, ,\\
w^{2}_{d}(q_{k},q_{k+1})&=\frac{y_{k+1}-y_{k}}{h}-\frac{\theta_{k+1}-\theta_{k}}{h} \left( \frac{1}{2}\sin\left((1-\alpha)\varphi_{k}+\alpha\varphi_{k+1} \right) + \frac{1}{2}\sin\left(\alpha\varphi_{k}+(1-\alpha)\varphi_{k+1} \right)  \right)  \, ,
\end{align*}
then we still get the dynamics $\theta_{k+1}=2\theta_{k}-\theta_{k-1}$ for any $\alpha\in\left[ 0,1\right]$, and therefore we obtain an  isotropic submanifold $\operatorname{Im}(F_{d1}\times F_{d1})\circ \Gamma_d$.

Notice that if we take the map
\[
\bar{F}_{d1}(\theta_{k-1},\varphi_{k-1},x_{k-1},y_{k-1},\theta_{k},\varphi_{k}) = \left(\theta_{k-1},\varphi_{k-1},x_{k-1},y_{k-1},2(\theta_{k}-\theta_{k-1}),\varphi_{k}-\varphi_{k-1},0,0\right) \, , 
\]
instead of $F_{d1}$ then $\operatorname{Im}(\bar{F}_{d1}\times \bar{F}_{d1})\circ \Gamma_d$ is still an isotropic submanifold. This choice will appear in the next section.

Finally we consider the midpoint discretization of the constraints and the midpoint  discretization of the alternative $F_2$ given above, that is
\begin{multline*}
  F_{d2}(\theta_{k-1},\varphi_{k-1},x_{k-1},y_{k-1},\theta_k,\varphi_k) = \left( \theta_{k-1},\varphi_{k-1},x_{k-1},y_{k-1}, \frac{\theta_k-\theta_{k-1}}{\varphi_k-\varphi_{k-1}}, \right. \\
\frac{\varphi_k-\varphi_{k-1}}{h}-\frac{(\theta_k-\theta_{k-1})^2}{2(\varphi_k-\varphi_{k-1})^2}\Big(1+\cos\Big(\frac{\varphi_{k-1}+\varphi_k}{2}\Big)+\sin\Big(\frac{\varphi_{k-1}+\varphi_k}{2}\Big) \Big), 
\frac{\theta_k-\theta_{k-1}}{\varphi_k-\varphi_{k-1}},\frac{\theta_k-\theta_{k-1}}{\varphi_k-\varphi_{k-1}} \bigg) \, .
\end{multline*}
Then $\operatorname{Im}(F_{d2}\times F_{d2})\circ \Gamma_d$ becomes
\begin{multline*}
  \bigg( \theta_{k-1},\varphi_{k-1},x_{k-1},y_{k-1},\frac{\theta_k-\theta_{k-1}}{\varphi_k-\varphi_{k-1}},  
\frac{\varphi_k-\varphi_{k-1}}{h} 
\\
-\frac{(\theta_k-\theta_{k-1})^2}{2(\varphi_k-\varphi_{k-1})^2}\Big(1+\cos\Big(\frac{\varphi_{k-1}+\varphi_k}{2}\Big)+\sin\Big(\frac{\varphi_{k-1}+\varphi_k}{2}\Big) \Big), 
\frac{\theta_k-\theta_{k-1}}{\varphi_k-\varphi_{k-1}},\frac{\theta_k-\theta_{k-1}}{\varphi_k-\varphi_{k-1}}, 
\\
\theta_k, \varphi_k, x_{k-1}+h\cos\Big(\frac{\varphi_{k-1}+\varphi_k}{2}\Big)\frac{\theta_k-\theta_{k-1}}{h}, y_{k-1}+h\sin\Big(\frac{\varphi_{k-1}+\varphi_k}{2}\Big)\frac{\theta_k-\theta_{k-1}}{h}, 
\\
\frac{\theta_k - \theta_{k-1}}{\varphi_k-\varphi_{k-1}}, \frac{\varphi_k - \varphi_{k-1}}{h} - \frac{(\theta_k - \theta_{k-1})^2}{2(\varphi_k-\varphi_{k-1})}\Big( 1+\cos\Big( \frac{3\varphi_k - \varphi_{k-1}}{2} \Big) +\sin\Big( \frac{3\varphi_k - \varphi_{k-1}}{2} \Big) \Big), 
\\
\frac{\theta_k - \theta_{k-1}}{\varphi_k-\varphi_{k-1}}, \frac{\theta_k - \theta_{k-1}}{\varphi_k-\varphi_{k-1}} \bigg) \, , 
\end{multline*}
which is not an  isotropic submanifold.

\end{ex}

\subsection{Extension to a Lagrangian submanifold}

As pointed out in Remark \ref{remark-symplecto}, Definition \ref{variational-immersion} can be equivalently given by substituting the statement \emph{``$\operatorname{Im}(\gamma_{F,\Gamma_{d}})$ is an isotropic submanifold of $(T^{*}Q\times T^{*}Q,\Omega_Q)$''} by \emph{``$\operatorname{Im}(\Psi^{-1}\circ\gamma_{F,\Gamma_{d}})$ is an isotropic submanifold of $(T^*(Q\times Q),\omega_{Q\times Q})$''}.

Next we will show how to extend the isotropic submanifold $\operatorname{Im}(\Psi^{-1}\circ\gamma_{F,\Gamma_{d}})$ in order to obtain a Lagrangian one. The following lemma proved in \cite[Lemma 5.4]{BFM} is needed for this purpose: 

\begin{lemma} \label{lemma-isotropa}
Let $P$ be a smooth manifold, $C$ a submanifold of $P$ and $\gamma$ a section of $\left.T^{*}P\right|_{C}\longrightarrow C$, where $\left.T^{*}P\right|_{C}=\left\{\mu\in T^{*}P: \pi_{P}(\mu)\in C \right\}$ and $\pi_{P}:T^{*}P\longrightarrow P$ denotes the projection over $P$. If $\gamma(C)$ is isotropic in $(T^{*}P,\omega_{P})$, then there is a one-form $\tilde{\gamma}$ defined in a neighborhood of $C$ such that
\begin{enumerate}
\item $\left.\tilde{\gamma}\right|_{C}=\gamma$,
\item $d\tilde{\gamma}=0$.
\end{enumerate}		
\end{lemma}

Now if we take $P=Q\times Q$, $C=M_d$, and $\gamma=\Psi^{-1}\circ\gamma_{F,\Gamma_{d}}$, since $(\Psi^{-1}\circ\gamma_{F,\Gamma_{d}})(M_d)$ is isotropic in $(T^*(Q\times Q),\omega_{Q\times Q})$, then there is a one-form $\tilde{\gamma}$ defined in a neighborhood of $M_d$ such that $\left.\tilde{\gamma}\right|_{M_d}=\gamma$ and $d\tilde{\gamma}=0$. Then by the Poincar\'{e} lemma there is a locally defined function $L_d:Q\times Q\rightarrow \mathbb{R}$ such that  $\tilde{\gamma}=dL_d$.

Recall from \cite[Section 2]{BFM}, that in order to obtain a Lagrangian submanifold we need to choose $\dim(P)-\dim(\operatorname{Im}(\Psi^{-1}\circ\gamma_{F,\Gamma_{d}}))$ constraints that define  a submanifold $N\subset T^*(Q\times Q)$ such that $\operatorname{Im}(\Psi^{-1}\circ\gamma_{F,\Gamma_{d}}) \subset N$.
Next we compute the corresponding Hamiltonian vector fields (with respect to $\omega_{Q\times Q}$). If they are independent and not tangent to $\operatorname{Im}(\Psi^{-1}\circ\gamma_{F,\Gamma_{d}})$, we can extend the original manifold along its flows and obtain a Lagrangian submanifold \cite{Vaisman}, which depends on the choice of constraints. This method provides a source of (possibly) alternative Lagrangians. See \cite[Section 5]{BFM} for a derivation of alternative Lagrangians for the rolling disk using this technique in the continuous setting. 
Corresponding to the immersion $F_1$ we can obtain the Lagrangian function 
\[
L_1=\frac{1}{2}\left( \dot{\theta}^2+\dot{\varphi}^2-\dot{x}^2-\dot{y}^2 \right)+\dot{\theta}(\cos(\varphi)\dot{x}+\sin(\varphi)\dot{y}) \, ,
\]
while for $F_2$, and appropriate choice of constraints, we can obtain
\[
L_2=\frac{1}{2}\left(\dot{\varphi}^{2}-\dot{\theta}^{2}-\dot{x}^{2}-\dot{y}^{2}\right)+\frac{\dot{\theta}^{2}}{2\dot{\varphi}}\left( 1-\cos(\varphi)-\sin(\varphi) \right)+\dot{\theta}\dot{x}\left( \cos(\varphi)+\frac{1}{\dot{\varphi}} \right)+\dot{\theta}\dot{y}\left( \sin(\varphi)+\frac{1}{\dot{\varphi}} \right) \, .
\] 
We will now see an example of this process in the discrete case.

\begin{ex} \label{disk-ext}
Consider the vertical rolling disk again. With $\Gamma_d$ and $F_{d1}$ as in Example \ref{disk}, we obtain the isotropic submanifold $\operatorname{Im}(\Psi^{-1}\circ\gamma_{F_{d1},\Gamma_{d}})$ of $(T^*(Q\times Q),\omega_{Q\times Q})$ given by
  \begin{multline*}
    \left( \theta_{k-1}, \varphi_{k-1}, x_{k-1}, y_{k-1}, \theta_{k},\varphi_{k}, x_{k-1}+\cos\Big(\frac{\varphi_{k-1}+\varphi_{k}}{2}\Big)(\theta_{k}-\theta_{k-1}),
\right.
\\
\left.
y_{k-1}+\sin\Big(\frac{\varphi_{k-1}+\varphi_{k}}{2}\Big)(\theta_{k}-\theta_{k-1}), 2\frac{\theta_{k-1}-\theta_{k}}{h}, \frac{\varphi_{k-1}-\varphi_{k}}{h}, 0, 0, 2\frac{\theta_{k}-\theta_{k-1}}{h}, \frac{\varphi_{k}-\varphi_{k-1}}{h},0,0
\right) \, ,
  \end{multline*}
where we denote coordinates on $(T^*(Q\times Q),\omega_{Q\times Q})$ by 
\[(\theta_{k-1}, \varphi_{k-1}, x_{k-1}, y_{k-1}, \theta_{k},\varphi_{k},x_k,y_k,p_{\theta_{k-1}},p_{\varphi_{k-1}},p_{x_{k-1}},p_{y_{k-1}},p_{\theta_{k}},p_{\varphi_{k}},p_{x_k},p_{y_k}).\]
Now we can choose for instance the constraints 
\begin{align*}
\phi_1&= x_k-x_{k-1}-\cos\left(\frac{\varphi_{k-1}+\varphi_{k}}{2}\right)(\theta_k-\theta_{k-1})+p_{x_k} \, , \\
\phi_2&=y_k-y_{k-1}-\sin\left(\frac{\varphi_{k-1}+\varphi_{k}}{2}\right)(\theta_k-\theta_{k-1})+p_{y_k} \, ,
\end{align*}
with corresponding Hamiltonian vector fields
\begin{align*}
X_{\phi_1}&=\frac{\partial}{\partial p_{x_k}}-\frac{\partial}{\partial p_{x_{k-1}}}-\cos\left( \frac{\varphi_{k-1}+\varphi_k}{2} \right)\left( \frac{\partial}{\partial p_{\theta_k}}-\frac{\partial}{\partial p_{\theta_{k-1}}} \right) \\
&\mathrel{\phantom{=}}+\frac{\theta_k-\theta_{k-1}}{2}\sin\left( \frac{\varphi_{k-1}+\varphi_k}{2} \right)\left( \frac{\partial}{\partial p_{\varphi_k}}+\frac{\partial}{\partial p_{\varphi_{k-1}}} \right)-\frac{\partial}{\partial x_k} \, , \\
X_{\phi_2}&=\frac{\partial}{\partial p_{y_k}}-\frac{\partial}{\partial p_{y_{k-1}}}-\sin\left( \frac{\varphi_{k-1}+\varphi_k}{2} \right)\left( \frac{\partial}{\partial p_{\theta_k}}-\frac{\partial}{\partial p_{\theta_{k-1}}} \right) \\
&\mathrel{\phantom{=}}-\frac{\theta_k-\theta_{k-1}}{2}\cos\left( \frac{\varphi_{k-1}+\varphi_k}{2} \right)\left( \frac{\partial}{\partial p_{\varphi_k}}+\frac{\partial}{\partial p_{\varphi_{k-1}}} \right)-\frac{\partial}{\partial y_k} \, .
\end{align*}
If we extend along the flows of $X_{\phi_1}$ and $X_{\phi_2}$ we obtain the Lagrangian submanifold
\begin{multline*}
  \left(
\theta_{k-1}, \varphi_{k-1}, x_{k-1}, y_{k-1}, \theta_{k},\varphi_{k}, x_k, y_k,
\right.
\\
\left( -\frac{2}{h}+1 \right)(\theta_k-\theta_{k-1})-\cos\left( \frac{\varphi_{k-1}+\varphi_k}{2} \right)(x_k-x_{k-1})-\sin\left( \frac{\varphi_{k-1}+\varphi_k}{2} \right)(y_k-y_{k-1}),
\\
-\frac{\varphi_k-\varphi_{k-1}}{h}+\frac{\theta_k-\theta_{k-1}}{2}\left( \cos\left( \frac{\varphi_{k-1}+\varphi_k}{2} \right)(y_k-y_{k-1})-\sin\left( \frac{\varphi_{k-1}+\varphi_k}{2} \right)(x_k-x_{k-1}) \right),
\\
x_k-x_{k-1}-\cos\left( \frac{\varphi_{k-1}+\varphi_k}{2} \right)(\theta_k-\theta_{k-1}), y_k-y_{k-1}-\sin\left( \frac{\varphi_{k-1}+\varphi_k}{2} \right)(\theta_k-\theta_{k-1}),
\\
\left( \frac{2}{h}-1 \right)(\theta_k-\theta_{k-1})+\cos\left( \frac{\varphi_{k-1}+\varphi_k}{2} \right)(x_k-x_{k-1})+\sin\left( \frac{\varphi_{k-1}+\varphi_k}{2} \right)(y_k-y_{k-1}),
\\
\frac{\varphi_k-\varphi_{k-1}}{h}+\frac{\theta_k-\theta_{k-1}}{2}\left( \cos\left( \frac{\varphi_{k-1}+\varphi_k}{2} \right)(y_k-y_{k-1})-\sin\left( \frac{\varphi_{k-1}+\varphi_k}{2} \right)(x_k-x_{k-1}) \right),
\\
\left. -\left( x_k-x_{k-1}-\cos\left( \frac{\varphi_{k-1}+\varphi_k}{2} \right)(\theta_k-\theta_{k-1}) \right), -\left( y_k-y_{k-1}-\sin\left( \frac{\varphi_{k-1}+\varphi_k}{2} \right)(\theta_k-\theta_{k-1}) \right) \right) \, ,
\end{multline*}
with corresponding discrete Lagrangian
\begin{align} \label{Lagrangiano-discreto-raro}
L_d&=-\frac{1}{2}(x_k-x_{k-1})^2-\frac{1}{2}(y_k-y_{k-1})^2+\left(\frac{1}{h}-\frac{1}{2}\right)(\theta_k-\theta_{k-1})^2+\frac{1}{2h}(\varphi_k-\varphi_{k-1})^2  \\
&\mathrel{\phantom{=}}+(\theta_k-\theta_{k-1})\left(\cos\left( \frac{\varphi_{k-1}+\varphi_k}{2} \right)(x_k-x_{k-1})+\sin\left( \frac{\varphi_{k-1}+\varphi_k}{2} \right)(y_k-y_{k-1}) \right) \, . \nonumber
\end{align}
The DEL equations corresponding to $L_d$ are
\begin{align}
x_{k-1}-2x_k+x_{k+1}+(\theta_k-\theta_{k-1})\cos\left( \frac{\varphi_{k-1}+\varphi_k}{2}\right)-(\theta_{k+1}-\theta_k)\cos\left( \frac{\varphi_k+\varphi_{k+1}}{2}\right) &=0 \, ,  \label{d1} \\
y_{k-1}-2y_k+y_{k+1}+(\theta_k-\theta_{k-1})\sin\left( \frac{\varphi_{k-1}+\varphi_k}{2}\right)-(\theta_{k+1}-\theta_k)\sin\left( \frac{\varphi_k+\varphi_{k+1}}{2}\right) &=0 \, , \label{d2}  \\
\left(\frac{-2+h}{h}\right)(\theta_{k-1}-\theta_k)-\left(\frac{-2+h}{h}\right)(\theta_k-\theta_{k+1}) \hspace{15em}& \nonumber \\
+(x_k-x_{k-1})\cos\left( \frac{\varphi_{k-1}+\varphi_k}{2}\right)+(x_k-x_{k+1})\cos\left( \frac{\varphi_k+\varphi_{k+1}}{2}\right) & \nonumber \\
+(y_k-y_{k-1})\sin\left( \frac{\varphi_{k-1}+\varphi_k}{2}\right)+(y_k-y_{k+1})\sin\left( \frac{\varphi_k+\varphi_{k+1}}{2}\right) &=0 \, , \label{d3}  \\
\frac{\varphi_k-\varphi_{k-1}}{h}+\frac{\varphi_k-\varphi_{k+1}}{h} \hspace{25em}& \nonumber \\
+\frac{1}{2}(\theta_{k-1}-\theta_k)\left( (y_{k-1}-y_k)\cos\left( \frac{\varphi_{k-1}+\varphi_k}{2}\right)+(x_k-x_{k-1})\sin\left( \frac{\varphi_{k-1}+\varphi_k}{2}\right) \right) & \nonumber\\
+\frac{1}{2}(\theta_k-\theta_{k+1})\left((y_k-y_{k+1})\cos\left( \frac{\varphi_k+\varphi_{k+1}}{2}\right)+(x_{k+1}-x_k)\sin\left( \frac{\varphi_k+\varphi_{k+1}}{2}\right) \right) &=0  \, . \label{d4} 
\end{align}
When restricted to the constraint submanifold given by
\begin{align*}
x_k&=x_{k-1}+\cos\left(\frac{\varphi_{k-1}+\varphi_{k}}{2}\right)(\theta_k-\theta_{k-1}) \, , \\
y_k&=y_{k-1}+\sin\left(\frac{\varphi_{k-1}+\varphi_{k}}{2}\right)(\theta_k-\theta_{k-1}) \, ,
\end{align*}
Equations (\ref{d1}) and (\ref{d2}) identically vanish and Equations (\ref{d3}) and (\ref{d4}) become $\theta_{k+1}=2\theta_k-\theta_{k-1}$ and $\varphi_{k+1}=2\varphi_k-\varphi_{k-1}$ respectively. Hence we recover the SOdE in Example \ref{disk}.

\end{ex}

\begin{remark} \label{remark-sin-h}
Consider the Lagrangian obtained in \cite{BFM} by extension of an isotropic submanifold corresponding to $F_1:(\theta,\varphi,x,y,\dot{\theta},\dot{\varphi}) \longmapsto (\theta,\varphi,x,y,2\dot{\theta},\dot{\varphi},0,0)$, given by
\[
L_1=\frac{1}{2}\left( \dot{\theta}^2+\dot{\varphi}^2-\dot{x}^2-\dot{y}^2 \right)+\dot{\theta}(\cos(\varphi)\dot{x}+\sin(\varphi)\dot{y}) \, .
\]
If we take the discretization 
\begin{align*}
L_{d1}&=\frac{1}{2}\left( \left(\frac{\theta_k-\theta_{k-1}}{h}\right)^2+\left(\frac{\varphi_k-\varphi_{k-1}}{h}\right)^2-\left(\frac{x_k-x_{k-1}}{h}\right)^2-\left(\frac{y_k-y_{k-1}}{h}\right)^2 \right) \\
&\mathrel{\phantom{=}} +\frac{\theta_k-\theta_{k-1}}{h}\left(\cos\left(\frac{\varphi_k+\varphi_{k-1}}{2}\right)\frac{x_k-x_{k-1}}{h}+\sin\left(\frac{\varphi_k+\varphi_{k-1}}{2}\right)\frac{y_k-y_{k-1}}{h}\right) \, , 
\end{align*}
then the DEL equations are $\theta_{k+1}=2\theta_k-\theta_{k-1}$ and $\varphi_{k+1}=2\varphi_k-\varphi_{k-1}$ when restricted to the constraint submanifold.

If instead of $F_{d1}$ we consider $\bar{F}_{d1}$ in Example \ref{disk}, then by choosing the same constraints $\phi_1$ and $\phi_2$ as in Example \ref{disk-ext}, and extending the isotropic submanifold $\operatorname{Im}(\Psi^{-1}\circ\gamma_{\bar{F}_{d1},\Gamma_d})$, we obtain the discrete Lagrangian
\begin{align*}
\bar{L}_d&=\frac{h^2}{2}\left( \left(\frac{\theta_k-\theta_{k-1}}{h}\right)^2+\left(\frac{\varphi_k-\varphi_{k-1}}{h}\right)^2-\left(\frac{x_k-x_{k-1}}{h}\right)^2-\left(\frac{y_k-y_{k-1}}{h}\right)^2 \right. \\
&\mathrel{\phantom{=}} \left. +\frac{\theta_k-\theta_{k-1}}{h}\left(\cos\left(\frac{\varphi_k+\varphi_{k-1}}{2}\right)\frac{x_k-x_{k-1}}{h}+\sin\left(\frac{\varphi_k+\varphi_{k-1}}{2}\right)\frac{y_k-y_{k-1}}{h}\right) \right) \\
&= h^2 L_{d1} \, . 
\end{align*}

\end{remark}

We have run simulations of the vertical rolling disk using the DLA integrator (\ref{eqn:EL})-(\ref{eqn:DLA2}). We have used several alternative discretizations for defining the discrete constraints $w_d^a$:
\begin{itemize}
  \item Midpoint rule: $w_d^a(q_k,q_{k+1})=w^a\left( \frac{q_k+q_{k+1}}{2}\right)\left( \frac{q_{k+1}-q_k}{h}\right)$;
  \item Trapezoidal rule: 	$w_d^a(q_k,q_{k+1})= \frac{1}{2}\left(w^a\left( q_k\right)\left( \frac{q_{k+1}-q_k}{h}\right)+
w^a\left( q_{k+1}\right)\left( \frac{q_{k+1}-q_k}{h}\right) \right);$
  \item $\alpha$-trapezoidal rule:
	\begin{align*}
	w_d^a(q_k,q_{k+1}) &= \frac{1}{2} \left( w^a\left( (1-\alpha)q_k+\alpha q_{k+1}\right)\left( \frac{q_{k+1}-q_k}{h}\right) 
	\right.   \\
 &\mathrel{\phantom{=}} \left. + w^a\left( \alpha q_k + (1-\alpha)q_{k+1}\right)\left( \frac{q_{k+1}-q_k}{h}\right) \right) \, ,
  \end{align*}
which reduces to the trapezoidal rule for $\alpha=0$ and $\alpha=1$, and to the midpoint rule for $\alpha=1/2$;
  \item Euler A: $w_d^a(q_k,q_{k+1})= w^a\left( q_k\right)\left( \frac{q_{k+1}-q_k}{h}\right)$;
  \item Euler B: $w_d^a(q_k,q_{k+1})= w^a\left( q_{k+1}\right)\left( \frac{q_{k+1}-q_k}{h}\right)$.
\end{itemize}

For the following choices of a Lagrangian function $L$, we have computed numerically the values of the energy $K=({\partial  L}/{\partial \dot q}) \dot q- L$ along the solutions:
\begin{itemize}
\item $L_1=\frac{1}{2}\left(\dot\theta^2+\dot \varphi^2- \dot x^2 - \dot y^2\right)+ \dot \theta(\cos(\varphi) \dot x+\sin(\varphi) \dot y)$ (see Remark \ref{remark-sin-h}), which gives $K_1=L_1$,
\item $L_2=\frac{1}{2}\left(-\dot\theta^2+\dot \varphi^2- \dot x^2 - \dot y^2\right)+ \frac{ \dot \theta^2}{2 \dot \varphi}(1-\cos(\varphi)-\sin(\varphi))+ \dot \theta \dot x \left(\cos(\varphi)+\frac{1}{ \dot \varphi}\right)+ \dot \theta \dot y \left(\sin(\varphi)+\frac{1}{ \dot \varphi}\right)$ (see Example \ref{disk}) which gives $K_2=\frac{1}{2}\left(-\dot\theta^2+\dot \varphi^2- \dot x^2 - \dot y^2\right)+ \dot \theta(\cos(\varphi) \dot x+\sin(\varphi) \dot y)$,
\item $L_3=h^2\left(-\frac{1}{2}\dot x^2-\frac{1}{2}\dot y^2+\left(\frac{1}{h}-\frac{1}{2}\right)\dot \theta^2+\frac{1}{2h}\dot \varphi^2+ \dot \theta(\cos(\varphi)  \dot x+\sin(\varphi)  \dot y)\right)$, whose corresponding midpoint discretization is (\ref{Lagrangiano-discreto-raro}), which gives $K_3=L_3$.
\end{itemize}
The energy functions $K_1$, $K_2$ and $K_3$ were discretized using the midpoint rule to obtain $K^d_1$, $K^d_2$, $K^d_3$. The results of the simulations for all methods, except for Euler A and B, preserved the energy functions, up to numerical truncation errors. For the $\alpha$-trapezoidal discretization, all the values of $\alpha$ that we have used preserve the energy functions. This is expected because we already saw in Example \ref{disk} that for any $\alpha$ we obtain a variational SOdE. 

Note that $K_1$ and $K_2$ only differ in the sign of the term $\frac{1}{2} \dot \theta ^2$, whose discrete version is $\frac{1}{2h^2}(\theta_{k+1}-\theta_k)^2$. For all $\alpha$, one of the discrete evolution equations is $\theta_{k+1}=2\theta_k-\theta_{k-1}$, so $K^d_2-K^d_1$ is constant along solutions.
Similarly, it is easy to show that the preservation of either $K_2^d$ or $K_3^d$ along solutions implies the preservation of the other one. Indeed,
\[
K_2-\frac{K_3}{h^2}=-\frac{\dot\theta^2}{h}+\frac{h-1}{2h}\dot\varphi^2,
\]
and $\theta_k$ and $\varphi_k$ evolve uniformly, that is, both $\theta_{k+1}-\theta_k$ and $\varphi_{k+1}-\varphi_k$ are constant. This implies that $K^d_2-{K^d_3}/{h^2}$ is constant.

The energy behavior of the Euler A and B discretizations is shown in Figure \ref{fig:EulerAB-rolling-disk}.

\begin{figure}[ht]
\begin{center}
\includegraphics[trim = 35mm 5mm 30mm 0mm, clip, scale=.58]{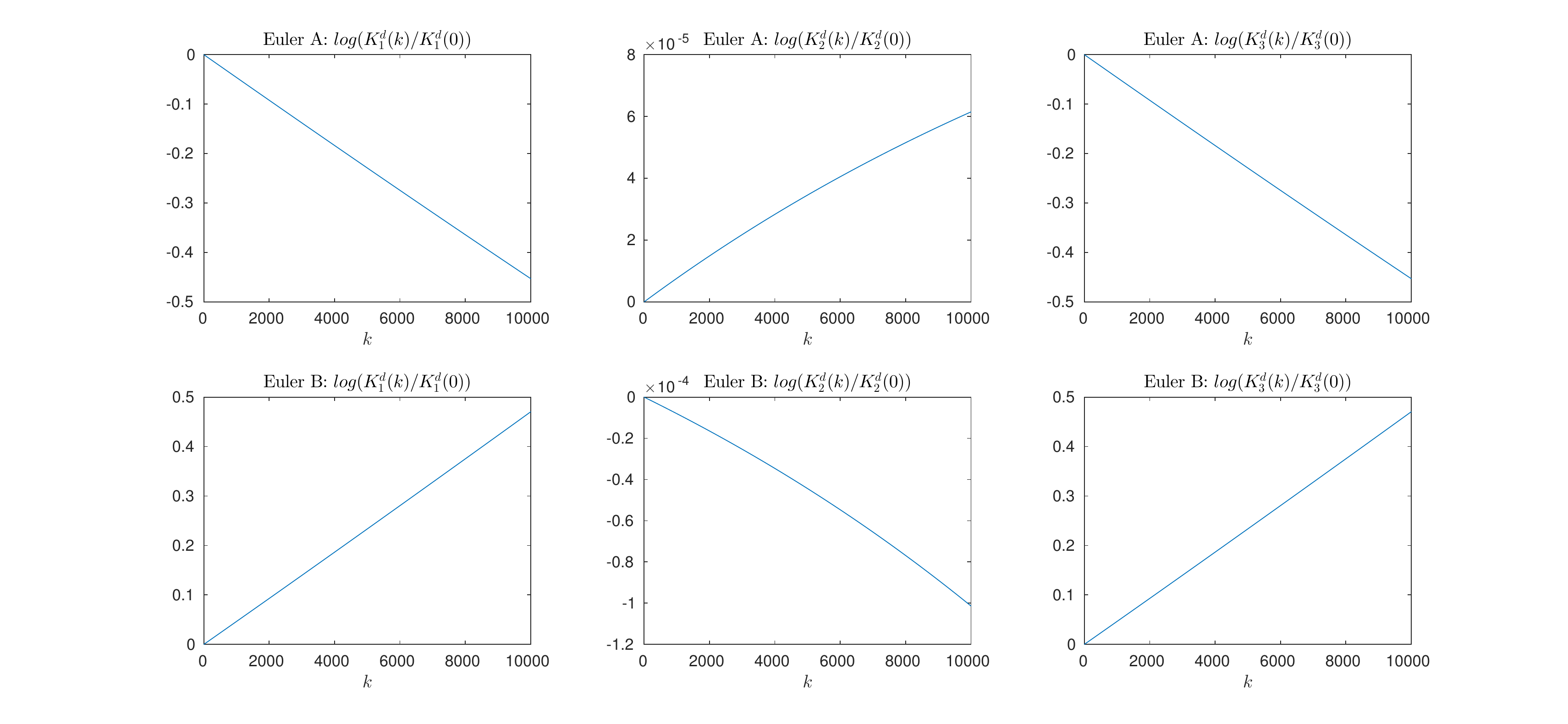}
\end{center}
\caption{Energy behavior for Euler A and B, vertical rolling disk. $T=500$, $h=0.05$, $(x_0,y_0,\theta_0,\varphi_0)=(1,1,0.5,0.3)$, $\theta_1=0.525$, $\varphi_1=0.31$; $x_1$ and $y_1$ satisfying the discrete constraints.}
\label{fig:EulerAB-rolling-disk}
\end{figure}

\section{Conclusions and future work} \label{sec7}
In this paper we have provided a formulation of the inverse problem for discrete second order systems. This formulation corresponds to the multiplier version of the classical inverse problem, since the goal consists in determining the equivalence of solutions of the given system and a system of discrete Euler-Lagrange equations. In this setting we can recover previous results known for the continuous problem, for instance a characterization of the problem given in terms of the existence of a Poincar\'{e}-Cartan two-form, see \cite{marmo1980,81Crampin}, and a characterization given in terms of the  existence of a Legendre transformation, see \cite{BFM}.
We have also analyzed the relationship between the continuous and discrete problems, the derivation of constants of motion from alternative Lagrangian formulations and the constrained case.

In the future we intend to study the following questions, related to the results of the present paper:
\begin{itemize}
\item In \cite{PatrickCuell} the authors provide a complete 
error analysis for variational integrators of regular Lagrangian systems, proving that if  we take as a discrete Lagrangian an approximation of a given order
 of the exact discrete Lagrangian then the  derived Hamiltonian discrete scheme is an approximation of the same order 
to the continuous flow.  We think that using our results in Subsection \ref{fromdc} it would be possible to study the converse of this important result under some regularity conditions.   

\item If we need to implement  an integrator for a given second order differential equation, it would be useful to preserve the geometric invariants of this SODE as much as possible. For instance, if we know that the SODE is variational then there exists an energy that is preserved along the evolution and additionally we have symplecticity. Therefore, in this case, it is useful to use symplectic  or variational integrators  to avoid spurious non-physical effects in the simulations. 
As an example, given a Chaplygin system, the DLA algorithm can be reduced to an algorithm, called RDLA, on the quotient space $Q/G$ by a Lie group, provided that we choose  the discrete Lagrangian $L_d$ and the discrete constraint space $D_d$ to be invariant under the diagonal action of $G$ on $Q\times Q$ \cite{CM2001}. It is in general of the form
\[
D_1L_d^*(r_k,r_{k+1})+D_2L_d^*(r_{k-1},r_k)=F^{-}(r_k,r_{k+1})+F^{+}(r_{k-1},r_k) \, .
\]
Under some extra assumptions we get $F^{-}(r_k,r_{k+1})=F^{+}(r_{k-1},r_k)=0$, and therefore the RDLA algorithm gives a variational integrator on $Q/G$, but this is not generally the case. It would be interesting to know if an alternative Lagrangian can be found so that the RDLA integrator and the HDEL and PTHDEL methods proposed in \cite{FBO2012} are variational.

\item Many mechanical systems are defined not on tangent bundles but on quotients by a symmetry Lie group, and therefore  the equations of motion are not the standard Euler-Lagrange equations. These equations often appear in a reduced version and it is possible to analyze the existence of a possible Lagrangian formulation using the geometrical setting of  Lie algebroids. 
 In \cite{BFM2} we study the inverse problem for SODEs defined on Lie algebroids.  In a future paper, we will study the discrete case using the formalism of Lie groupoids \cite{MMM06Grupoides}.

\end{itemize}

\section*{Acknowledgments}

This work has been partially supported by research grants MTM2013-42870-P, MTM2016-76702-P (MINECO) and the ICMAT Severo Ochoa project SEV-2015-0554 (MINECO). 
MFP has been financially supported by a FPU scholarship from MECD.
SF has been supported by CONICET Argentina (PIP 2013-2015 GI
11220120100532CO), ANPCyT Argentina (PICT 2013-1302) and SGCyT UNS.

\bibliographystyle{plain}
\bibliography{References}

\end{document}